\documentclass[11pt]{amsart}%
\usepackage{amsmath}
\usepackage{amssymb}
\usepackage{amsthm}
\usepackage[mathscr]{eucal}
\usepackage{mathrsfs}
\usepackage{psfrag}
\usepackage{fullpage}
\usepackage{color}
\usepackage{eucal}
\usepackage[dvips]{graphicx}
\usepackage{amsfonts}%
\usepackage{amsaddr}
\providecommand{\U}[1]{\protect\rule{.1in}{.1in}}
\newtheorem{proposition}{Proposition}[section]
\newtheorem{theorem}[proposition]{Theorem}

\newtheorem{lemma}[proposition]{Lemma}

\newtheorem{definition}[proposition]{Definition}
\newtheorem{remark}[proposition]{Remark}

\newtheorem{condition}[proposition]{Condition}

\numberwithin{equation}{section}
\numberwithin{proposition}{section}

\begin{document}
\title{Large Deviations and Importance Sampling for Systems of Slow-Fast Motion}
\author{Konstantinos Spiliopoulos}
\address{Division of Applied Mathematics, Brown University, Providence, RI 02912 \& Department of Mathematics and Statistics, Boston University, Boston, MA, 02215}
\email{kspiliop@dam.brown.edu,kspiliop@math.bu.edu}
\thanks{The author would like to acknowledge support by the Department of Energy (DE-SCOO02413) during his stay at Brown University and by a start-up fund by Boston University during completion and revision of this paper.}
\date{\today}
\maketitle

\begin{abstract}
In this paper we develop the large deviations principle and a
rigorous mathematical framework for asymptotically efficient  importance
sampling schemes for general, fully dependent systems of stochastic
differential
 equations of slow and fast motion with small noise in the slow component. We assume  periodicity with respect to  the fast component. Depending on the interaction
 of the fast scale with the smallness of the noise, we get different behavior. We examine how one range
of interaction differs from the other one both for the large deviations and for the importance sampling. We use the large deviations results to identify  asymptotically optimal importance sampling schemes in each case.
Standard Monte Carlo schemes perform poorly in the small noise limit. In the presence of multiscale aspects one
 faces additional
 difficulties and straightforward adaptation of importance sampling schemes for standard small noise diffusions will
 not produce efficient schemes. It turns out that one has to consider the so called cell problem from the
 homogenization theory
 for Hamilton-Jacobi-Bellman equations in order to guarantee asymptotic optimality. We use stochastic control arguments.
\end{abstract}

\textbf{Keywords}: importance sampling, Monte Carlo, large deviations, homogenization, multiscale, slow-fast motion.

\textbf{AMS}: 60F05, 60F10, 60G60

\section{Introduction}
Let us consider the $m+(d-m)$ dimensional process $(X^{\epsilon},Y^{\epsilon})=\{(X^{\epsilon}(s),Y^{\epsilon}(s)), 0\leq s\leq T\}$ satisfying the system of stochastic differential equations (SDE's)

\begin{eqnarray}
dX^{\epsilon}(s)&=&\left[  \frac{\epsilon}{\delta}b\left(  X^{\epsilon}(s)%
,Y^{\epsilon}(s)\right)+c\left(  X^{\epsilon}(s)%
,Y^{\epsilon}(s)\right)\right]   ds+\sqrt{\epsilon}%
\sigma\left(  X^{\epsilon}(s),Y^{\epsilon}(s)\right)
dW(s),\nonumber\\
dY^{\epsilon}(s)&=&\frac{1}{\delta}\left[  \frac{\epsilon}{\delta}f\left(  X^{\epsilon}(s)%
,Y^{\epsilon}(s)\right)  +g\left(  X^{\epsilon}(s)%
,Y^{\epsilon}(s)\right)\right] ds+\frac{\sqrt{\epsilon}}{\delta}\left[
\tau_{1}\left(  X^{\epsilon}(s),Y^{\epsilon}(s)\right)\right.
dW(s)+\nonumber\\
& &\hspace{8cm}\left.+\tau_{2}\left(  X^{\epsilon}(s),Y^{\epsilon}(s)\right)dB(s)\right], \label{Eq:Main}\\
X^{\epsilon}(0)&=&x_{0},\hspace{0.2cm}Y^{\epsilon}(0)=y_{0}\nonumber
\end{eqnarray}
where $\delta=\delta(\epsilon)\downarrow0$ as $\epsilon\downarrow0$ and $(W(s), B(s))$ is a $2\kappa-$dimensional standard Wiener process. The functions $b(x,y),c(x,y),\sigma(x,y), f(x,y), g(x,y), \tau_{1}(x,y)$ and $\tau_{2}(x,y)$ are assumed to be sufficiently smooth (see Condition \ref{A:Assumption1}) and periodic with period $\lambda$ in every direction with respect to the second variable.

One can interpret the system of (\ref{Eq:Main}) as a system of slow and fast motion with $X^{\epsilon}$ playing the role of the slow motion and
$Y^{\epsilon}$ playing the role of the fast motion. The goal of this paper is to provide a large deviations analysis of (\ref{Eq:Main}) that allows to rigorously
develop  the importance sampling theory for estimation of functionals such as
\begin{equation}
\theta(\epsilon)\doteq\mathrm{E}[e^{-\frac{1}{\epsilon}h(X^{\epsilon}%
(T))}|X^{\epsilon}(0)=x_{0}, Y^{\epsilon}(0)=y_{0}].\label{Eq:ToBeEstimated}
\end{equation}
Importance sampling is a variance reduction technique in Monte Carlo simulation. As it is well
known, standard Monte Carlo sampling techniques perform very poorly in that
the relative errors under a fixed computational effort grow rapidly as the
event becomes more and more rare. Estimating rare event probabilities in the
context of slow-fast systems presents extra difficulties due to the underlying fast motion and its
interaction with the intensity of the noise $\epsilon$.

Depending on the order that $\epsilon,\delta$ go to zero, we have three different regimes of interaction:
\begin{equation}
\lim_{\epsilon\downarrow0}\frac{\epsilon}{\delta}=%
\begin{cases}
\infty & \text{Regime 1,}\\
\gamma\in(0,\infty) & \text{Regime 2,}\\
0 & \text{Regime 3.}%
\end{cases}
\label{Def:ThreePossibleRegimes}%
\end{equation}

If $\delta$ goes to zero faster than $\epsilon$ (Regime 1) then
homogenization occurs first, whereas if $\epsilon$ goes to zero
faster than $\delta$ (Regime 3) then large deviations theory tells
how quickly  (\ref{Eq:Main}) converges to the averaged deterministic
ODE given by setting $\epsilon$ equal to zero. If the two parameters
go to zero together then one has an intermediate situation (Regime 2).

The study of rare events in the
multiscale context is a difficult problem due to the presence
of the underlying fast motion. The first necessary step is to
develop the associated large deviations theory.
Using weak convergence arguments the authors in
 \cite{DupuisSpiliopoulos} prove the large deviations principle for the special case $f=b,g=c,\tau_{1}=\sigma$ and $\tau_{2}=0$.
 We extend the results of \cite{DupuisSpiliopoulos} to the current general setup.
Then, using the large deviation results and stochastic control
arguments  we construct asymptotically
optimal importance sampling schemes with rigorous bounds on
performance. The construction is based on subsolutions for an
associated Hamilton-Jacobi-Bellman (HJB) equation as in
\cite{DupuisSpiliopoulosWang, DupuisWang2}. The situation here is
complicated due to the presence of the fast motion. It turns out
that changes of measure that are implied by the homogenized system
do not lead to efficient importance sampling schemes. The standard
arguments have to be modified taking into account the solution to
the related "cell problem" which is
 different for each regime. This is also tightly related to the homogenization theory for HJB equations. A control in
full feedback form, i.e., a function of both the slow variable $X^{\epsilon}$
and the fast variable $Y^{\epsilon}$, is used to construct dynamic importance sampling schemes with precise asymptotic performance bounds. The control involves both the solution to the appropriate homogenized HJB equation and to its corresponding cell problem.

The novelty of this work lies in developing (a) the large deviations
principle and (b) a general and rigorous mathematical
 framework for the study of importance sampling schemes for systems of slow-fast motion as in (\ref{Eq:Main}) for all three regimes of interaction,
(\ref{Def:ThreePossibleRegimes}). Multiscale stochastic control
problems and related large deviations problems have been studied
elsewhere as well under various assumptions and dependencies of the
coefficients of the system on the slow and fast motion, see
\cite{Baldi, BorkarGaitsgory, DupuisSpiliopoulos, FengFouqueKumar,
FS, Kushner, Kushner1, Lipster, PardouxVeretennikov1, Veretennikov,
VeretennikovSPA2000}. The papers \cite{FengFouqueKumar, FS,
Kushner1, Veretennikov, VeretennikovSPA2000} address the large
deviations principle for Regime $2$ for special cases of dependence
of the coefficients on $(x,y)$. With the exception of
\cite{DupuisSpiliopoulos,Kushner1}, they express it through a Legendre-Fenchel
transform of the limit of the normalized logarithm of an exponential
moment or of the first eigenvalue of an associated operator. Here we
provide an explicit characterization of the action functional. Also, the large deviations
arguments in the aforementioned papers do not cover the full nonlinear case that we study here and do not  seem
to provide insights into how to construct asymptotically efficient importance sampling
schemes.  Some related importance sampling results on this problem
have been recently obtained in \cite{DupuisSpiliopoulosWang}. There
the authors study the special case of $f=b,g=c,\tau_{1}=\sigma$ and
$\tau_{2}=0$ for Regime 1 only and provide simulation studies for
that particular case as well. It is also demonstrated there that straightforward adaptation of importance sampling schemes for standard diffusion processes (without the multiscale aspect) will have poor results in the multiscale setting. This translates in that one needs to consider the solution to the cell problem in problems with multiple scales in order to guarantee good asymptotic performance.  The treatment of the general case, that is the content of the current paper,
requires additional considerations. In particular, the
identification of the optimal control and of the associated
subsolutions and cell problems are more involved here even for
Regime $1$. The case of Regimes $2$ and $3$ is studied in this paper
for the first time. This work is closely related to the
homogenization theory of HJB equations, e.g.,
\cite{AlvarezBardi2001,ArisawaLions,BuckdahnIchihara,Evans2,HorieIshii,LionsSouganidis},
see Section \ref{S:HJB}.

We note here that one may possibly be able to relax the periodicity assumption both for the large deviations and for the importance sampling. In particular, in the case of Regime $1$ using
 the results and methodology of \cite{PardouxVeretennikov1,DupuisSpiliopoulos}  and of the present paper, we can probably prove an analogous result when the fast
variable takes values in $\mathbb{R}^{d-m}$ instead of the torus. Of course, one would need to impose the appropriate recurrence conditions for the fast motion. In the case of Regime $2$, the extension to the whole space with full dependence of the coefficients on $(x,y)$
is more involved. However, it  seems plausible that the methods of the current paper  can be combined  with those of
\cite{KaiseSheu,BensoussanFrehse,DupuisSpiliopoulos} to weaken the periodicity assumption for Regime $2$ as well. This will be addressed elsewhere.

The need to simulate rare events occurs in many application areas
including telecommunication, finance, insurance and chemistry. We
present some examples in Section \ref{S:Examples}. A model of interest in
chemical physics and chemistry is the first order Langevin equation
in a rough potential, e.g.
\cite{LifsonJackson,MondalGhosh,SavenWangWolynes,
DupuisSpiliopoulosWang2, Zwanzig}. This is a special case of the
system (\ref{Eq:Main}) with $f=b=-\nabla Q(y)$, $g=c=-\nabla V(x)$,
$\tau_{1}=\sigma=\textrm{constant}$ and $\tau_{2}=0$ and is
discussed in Subsection \ref{SS:Example1}. Another example,
discussed in Subsection \ref{SS:Example2}, is related to short time asymptotics of a process that depends on another fast mean reverting process.

The rest of the paper is organized as follows. In Section \ref{S:Notation} we introduce necessary notation and our assumptions.
Section \ref{S:LDP} is devoted to the related large deviations theory.
In Section \ref{S:IS} we develop the importance sampling theory for all three possible
regimes of interaction that guarantees asymptotic optimality. In Section \ref{S:HJB} we discuss the connection of the importance sampling theory with the homogenization of HJB equations. We conclude with
 Section \ref{S:Examples} where we examine how our results look like in some special cases of interest.

\section{Notation and assumptions}\label{S:Notation}

In this section we establish some notation and lay out our main assumptions. Let us assume a filtered probability space $(\Omega,\mathfrak{F}%
,\mathbb{P})$ equipped with a filtration $\mathfrak{F}_{t}$ that satisfies the
usual conditions, namely, $\mathfrak{F}_{t}$ is right continuous and
$\mathfrak{F}_{0}$ contains all $\mathbb{P}$-negligible sets.

The main assumption for the coefficients of (\ref{Eq:Main}) is as follows.

\begin{condition}
\label{A:Assumption1}

\begin{enumerate}
\item The functions $b(x,y),c(x,y),\sigma(x,y), f(x,y), g(x,y), \tau_{1}(x,y)$ and $\tau_{2}(x,y)$ are 
bounded in both variables and periodic with period $\lambda$ in the second variable
in each direction. We additionally assume that they
are $C^{1}(\mathbb{R}^{d-m})$ in $y$ and $C^{2}(\mathbb{R}^{m})$ in $x$ with all
partial derivatives continuous and globally bounded in $x$ and $y $.

\item The diffusion matrices $\sigma\sigma^{T}$ and $\tau_{1}\tau_{1}^{T}+\tau_{2}\tau_{2}^{T}$ are uniformly nondegenerate.

\end{enumerate}
\end{condition}

Under Condition  \ref{A:Assumption1} the system (\ref{Eq:Main}) has
a unique strong solution. The smoothness assumptions are stronger
than necessary, but they guarantee smoothness and boundedness of the
associated cell problems that will appear in the development of the
importance sampling theory. For notational convenience we define the
operator $\cdot:\cdot$, where for two matrices
$A=[a_{ij}],B=[b_{ij}]$
\[
A:B\doteq\sum_{i,j}a_{ij}b_{ij}.
\]

Let $\mathcal{Y}=\mathbb{T}^{d-m}$ be the $(d-m)$-dimensional torus. This is the state space of the fast motion. For the purposes of consistency with the related literature we use similar notation as in \cite{DupuisSpiliopoulos,DupuisSpiliopoulosWang} with the appropriate modifications in order to cover the more general set-up that we treat here.

Under Regime $1$, we also impose the following condition.

\begin{condition}
\label{A:Assumption2} Let $F\in\mathcal{C}^{2}\left(\mathcal{Y};\mathbb{R}\right)$ and consider the operator
\begin{equation*}
\mathcal{L}_{x}^{1}F(y)=f(x,y)\cdot\nabla_{y}F(y)+\frac{1}{2}\left(\tau_{1}\tau_{1}^{T}+\tau_{2}\tau_{2}^{T}\right)(x,y):\nabla_{y}\nabla_{y}F(y) \label{OperatorRegime1}%
\end{equation*}
equipped with periodic boundary conditions in $y$. Under Regime 1, we assume the centering condition
(see \cite{BLP}):
\[
\int_{\mathcal{Y}}b(x,y)\mu(dy|x)=0,
\]
where $\mu(dy|x)$ is the unique invariant measure
corresponding to the operator $\mathcal{L}_{x}^{1}$.
\end{condition}

Under Conditions \ref{A:Assumption1} and \ref{A:Assumption2}, Theorem  $3.3.4$ in \cite{BLP} guarantees
that for
each $\ell\in\{1,\ldots,m\}$ there is a unique, twice differentiable, with all partial derivatives up to second order bounded, $\lambda-$periodic in each direction function
$\chi_{\ell}(x,y)$ that satisfies
the cell problem:
\begin{equation}
\mathcal{L}_{x}^{1}\chi_{\ell}(x,y)=-b_{\ell}(x,y),\quad\int_{\mathcal{Y}}%
\chi_{\ell}(x,y)\mu(dy|x)=0. \label{Eq:CellProblem}%
\end{equation}
We write $\chi=(\chi_{1},\ldots,\chi_{m})$.

\vspace{0.4cm}

Let us denote $\mathcal{Z}=\mathbb{R}^{\kappa}$. This will be the space in which the control processes
that will appear in the next sections take values.

\begin{definition}
\label{Def:ThreePossibleOperators} For
$(x,y,z_{1},z_{2})\in\mathbb{R}^{m}\times\mathcal{Y}\times\mathcal{Z}\times\mathcal{Z}$
and for Regime $i=1,2,3$ defined in (\ref{Def:ThreePossibleRegimes})
we define the operators $\mathcal{L}_{z_{1},z_{2},x}^{i}$. For $i=1,2$ we let $\mathcal{D}(\mathcal{L}_{z,x}^{i})=\mathcal{C}%
^{2}(\mathcal{Y})$ and for $i=3$, $\mathcal{D}(\mathcal{L}_{z,x}%
^{3})=\mathcal{C}^{1}(\mathcal{Y})$. For
$F\in\mathcal{D}(\mathcal{L}_{z,x}^{i}) $ define
\begin{align}
\mathcal{L}_{x}^{1}F(y)  &  =f(x,y)\cdot\nabla_{y}F(y)+\frac{1}{2}\left(\tau_{1}\tau_{1}^{T}+\tau_{2}\tau_{2}^{T}\right)(x,y):\nabla_{y}\nabla_{y}F(y) \nonumber\\
\mathcal{L}_{z_{1},z_{2},x}^{2}F(y)  &  =\left[  \gamma
f(x,y)+g(x,y)+\tau_{1}(x,y)z_{1}+\tau_{2}(x,y)z_{2}\right]
\cdot\nabla_{y}F(y)+\gamma\frac{1}{2}\left(\tau_{1}\tau_{1}^{T}+\tau_{2}\tau_{2}^{T}\right)(x,y):\nabla_{y}%
\nabla_{y}F(y)\nonumber\\
\mathcal{L}_{z_{1},z_{2},x}^{3} F(y) &  =\left[
g(x,y)+\tau_{1}(x,y)z_{1}+\tau_{2}(x,y)z_{2}\right]  \cdot\nabla
_{y}F(y).\nonumber
\end{align}
\end{definition}


\begin{definition}
\label{Def:ThreePossibleFunctions} For $(x,y,z_{1},z_{2})\in\mathbb{R}^{m}\times\mathcal{Y}\times\mathcal{Z}\times\mathcal{Z}$
and for Regime $i=1,2,3$
defined in (\ref{Def:ThreePossibleRegimes}) we define the functions $\lambda_{i}(x,y,z_{1},z_{2}):\mathbb{R}%
^{m}\times\mathcal{Y}\times\mathcal{Z}\times\mathcal{Z}\rightarrow\mathbb{R}^{m}$ by
\begin{align}
\lambda_{1}(x,y,z_{1},z_{2})  &  = c(x,y)+\frac{\partial\chi}{\partial y}(x,y) g(x,y) +\sigma(x,y)z_{1}+\frac{\partial\chi}{\partial y}(x,y)\left(\tau_{1}(x,y)z_{1}+\tau_{2}(x,y)z_{2}\right) \nonumber\\
\lambda_{2}(x,y,z_{1},z_{2})  &  =\gamma b(x,y)+c(x,y)+\sigma(x,y)z_{1}\nonumber\\
\lambda_{3}(x,y,z_{1},z_{2})  &  =c(x,y)+\sigma(x,y)z_{1},\nonumber
\end{align}
where $\chi=(\chi_{1},\ldots,\chi_{m})$ is defined by (\ref{Eq:CellProblem}).
\end{definition}

For a Polish space $\mathcal{S}$, let $\mathcal{P}(\mathcal{S})$ be the space
of probability measures on $\mathcal{S}$. Next we recall the notion of viability as defined in \cite{DupuisSpiliopoulos}.

\begin{definition}
\label{Def:ViablePair} A pair $(\psi,\mathrm{P})\in\mathcal{C}%
([0,T];\mathbb{R}^{m})\times\mathcal{P}(\mathcal{Z}\times\mathcal{Z}\times\mathcal{Y}%
\times\lbrack0,T])$ will be called viable with respect to $(\lambda
,\mathcal{L})$ and write
$(\psi,\mathrm{P})\in\mathcal{V}_{(\lambda,\mathcal{L})}$, if the
following hold:
\begin{itemize}
 \item The function $\psi_{t}$ is absolutely continuous.
\item The measure $\mathrm{P}$
is square integrable in the sense that $\int_{\mathcal{Z}\times\mathcal{Z}\times\mathcal{Y}%
\times\lbrack0,T]}\left\Vert z\right\Vert
^{2}\mathrm{P}(dz_{1}dz_{2}dyds)<\infty$.
\item For all $t\in\lbrack0,T]$
\begin{equation}
\psi_{t}=x_{0}+\int_{\mathcal{Z}\times\mathcal{Z}\times\mathcal{Y}\times\lbrack0,t]}%
\lambda(\psi_{s},y,z_{1},z_{2})\mathrm{P}(dz_{1}dz_{2}dyds),
\label{Eq:AccumulationPointsProcessViable}%
\end{equation}
\item For all $t\in\lbrack0,T]$ and for every $f\in\mathcal{D}(\mathcal{L})$%
\begin{equation}
\int_{0}^{t}\int_{\mathcal{Z}\times\mathcal{Z}\times\mathcal{Y}}\mathcal{L}_{z_{1},z_{2},\psi_{s}%
}f(y)\mathrm{P}(dz_{1}dz_{2}dyds)=0, \label{Eq:AccumulationPointsMeasureViable}%
\end{equation}
\item For all $t\in\lbrack0,T]$
\begin{equation}
\mathrm{P}(\mathcal{Z}\times\mathcal{Z}\times\mathcal{Y}\times\lbrack0,t])=t.
\label{Eq:AccumulationPointsFullMeasureViable}%
\end{equation}

\end{itemize}

\end{definition}

Notice that equation (\ref{Eq:AccumulationPointsFullMeasureViable}) implies that the last
marginal of $\mathrm{P}$ is Lebesgue measure, and hence $\mathrm{P}$ can be
 decomposed in the form $\mathrm{P}(dz_{1}dz_{2}dydt)=\mathrm{P}_{t}(dz_{1}dz_{2}dy)dt$.

\section{Large deviations principle}\label{S:LDP}

The authors in \cite{DupuisSpiliopoulos} establish the large deviations principle related to (\ref{Eq:Main}) in the special case of  $f=b,g=c,\tau_{1}=\sigma$ and $\tau_{2}=0$. We extend the results of \cite{DupuisSpiliopoulos} to the current general setup.  A uniform approach to the large deviations problem for (\ref{Eq:Main}) is presented,
 allowing to
essentially treat all three regimes with the same general strategy, even though the technical details might be different from regime to regime. Moreover, in the course of the proof of the large deviations lower bound, we need to construct a nearly optimal control that attains the large deviations bound. As we will see in Section \ref{S:IS}, this control can guide the construction of efficient importance sampling for the estimation of quantities such as (\ref{Eq:ToBeEstimated}).

Essentially, in each regime, the
action functional is given by the infimization of a quadratic
functional, where the infimum is determined by the averaging of an
appropriate controlled version of the limiting slow motion with
respect to the corresponding fast motion. Both the limiting slow
motion and the fast motion with respect to which the averaging is
being done, differ from regime to regime. This is related to the
notion of viability from Definition \ref{Def:ViablePair} where the
viable pairs $(\lambda,\mathcal{L})$ are obtained from Definitions
\ref{Def:ThreePossibleOperators} and
\ref{Def:ThreePossibleFunctions} for each regime. What defers from the special case considered in \cite{DupuisSpiliopoulos} is the form of the appropriate viable pair
in each case. We present this
characterization below.

In preparation for stating the main large deviations results, we recall the concept of a
Laplace principle.

\begin{definition}
\label{Def:LaplacePrinciple} Let $\{X^{\epsilon},\epsilon>0\}$ be a family of
random variables taking values in a Polish space $\mathcal{S}$ and let $I$ be a rate function
on $\mathcal{S}$. We say that $\{X^{\epsilon},\epsilon>0\}$ satisfies the
Laplace principle with rate function $I$ if for every bounded and continuous
function $h:\mathcal{S}\rightarrow\mathbb{R}$
\begin{equation}
\lim_{\epsilon\downarrow0}-\epsilon\ln\mathbb{E}\left[  \exp\left\{
-\frac{h(X^{\epsilon})}{\epsilon}\right\}  \right]  =\inf_{x\in\mathcal{S}%
}\left[  I(x)+h(x)\right]. \label{Eq:LaplacePrinciple}%
\end{equation}

\end{definition}

If the level sets of the rate function (equivalently action functional) are compact, then the
Laplace principle is equivalent to the corresponding large deviations
principle with the same rate function (Theorems 2.2.1 and 2.2.3
in \cite{DupuisEllis}).

The derivation of the large deviations and importance sampling results are based on a variational representation for functionals of Wiener process derived in
\cite{BoueDupuis} that allows to rewrite the prelimit left hand side of (\ref{Eq:LaplacePrinciple}). Let $Z(\cdot)$ be a standard $n$-dimensional Wiener process and $F(\cdot)$ a bounded and measurable real-valued function define on the set of $\mathbb{R}^{n}-$ valued continuous functions on $[0,T]$. By Theorem 3.1 in \cite{BoueDupuis} we have

\begin{equation}
-\log\mathbb{E}\left[\exp\left\{-F(Z(\cdot))\right\}\right]=\inf_{u\in\mathcal{A}}\mathbb{E}\left[\frac{1}{2}\int_{0}^{T}\left\Vert u(s)\right\Vert^{2}ds+F\left(Z(\cdot)+\int_{0}^{\cdot}u(s)ds\right)\right]
\end{equation}
where $\mathcal{A}$ is the set of all $\mathfrak{F}_{s}-$progressively
measurable $n$-dimensional processes $u\doteq\{u(s),0\leq s\leq T\}$
satisfying
\begin{equation*}
\mathbb{E}\int_{0}^{T}\left\Vert u(s)\right\Vert ^{2}ds<\infty,
\label{A:AdmissibleControls}%
\end{equation*}
In the present case, let $Z(\cdot)=(W(\cdot),B(\cdot))$ and $n=2k$. Under Condition \ref{A:Assumption1}, the system has (\ref{Eq:Main}) has a unique strong solution. Therefore $X^{\epsilon}$ and $Y^{\epsilon}$ are measurable functions of $Z(\cdot)=(W(\cdot),B(\cdot))$. After setting $F(Z(\cdot))=h(X^{\epsilon}(\cdot))/\epsilon$ and rescaling the controls by $\frac{1}{\sqrt{\epsilon}}$ we get the representation

\begin{equation}
-\epsilon\ln\mathbb{E}_{x_{0}, y_{0}}
\left[  \exp\left\{  -\frac{h(X^{\epsilon}%
)}{\epsilon}\right\}  \right]  =\inf_{u\in\mathcal{A}}\mathbb{E}_{x_{0},y_{0}%
}\left[  \frac{1}{2}\int_{0}^{T}\left[\left\Vert u_{1}(s)\right\Vert ^{2}+\left\Vert u_{2}(s)\right\Vert ^{2}\right]ds+h(\bar
{X}^{\epsilon})\right] \label{Eq:VariationalRepresentation}
\end{equation}
where the pair $(\bar{X}^{\epsilon},\bar{Y}^{\epsilon})$ is the unique strong solution to

\begin{eqnarray}
d\bar{X}^{\epsilon}(s)&=&\left[  \frac{\epsilon}{\delta}b\left(  \bar{X}^{\epsilon}(s)%
,\bar{Y}^{\epsilon}(s)\right)+c\left(  \bar{X}^{\epsilon}(s)%
,\bar{Y}^{\epsilon}(s)\right)+\sigma\left(  \bar{X}_{t}^{\epsilon},\bar{Y}_{t}^{\epsilon}\right)  u_{1}(s)\right]   ds+\sqrt{\epsilon}%
\sigma\left(  \bar{X}^{\epsilon}(s),\bar{Y}^{\epsilon}(s)\right)
dW(s), \nonumber\\
d\bar{Y}^{\epsilon}(s)&=&\frac{1}{\delta}\left[  \frac{\epsilon}{\delta}f\left(  \bar{X}^{\epsilon}(s)
,\bar{Y}^{\epsilon}(s)\right)  +g\left(  \bar{X}^{\epsilon}(s)
,\bar{Y}^{\epsilon}(s)\right)+\tau_{1}\left(  \bar{X}^{\epsilon}(s)
,\bar{Y}^{\epsilon}(s)\right)u_{1}(s)+\tau_{2}\left(  \bar{X}^{\epsilon}(s)
,\bar{Y}^{\epsilon}(s)\right)u_{2}(s)\right]   ds\nonumber\\
& &\hspace{3.5cm}+\frac{\sqrt{\epsilon}}{\delta}\left[
\tau_{1}\left(  \bar{X}^{\epsilon}(s),\bar{Y}^{\epsilon}(s)\right)
dW(s)+\tau_{2}\left(  \bar{X}^{\epsilon}(s),\bar{Y}^{\epsilon}(s)\right)dB(s)\right],\label{Eq:Main2}\\
\bar{X}^{\epsilon}(0)&=&x_{0},\hspace{0.2cm}\bar{Y}^{\epsilon}(0)=y_{0}\nonumber
\end{eqnarray}

Therefore in order to derive the Laplace principle for
$\{X^{\epsilon}\}$, it is enough to study the limit of the right
hand side of the variational representation
(\ref{Eq:VariationalRepresentation}). The first step in doing so is
to consider the weak limit of the slow motion $\bar{X}^{\epsilon}$
of the controlled couple (\ref{Eq:Main2}). Due to the involved
controls, it is convenient to introduce the following occupation
measure. Let $\Delta=\Delta(\epsilon )\downarrow0$ as
$\epsilon\downarrow0$. The role of $\Delta(\epsilon)$ is to exploit
a time-scale separation. Let $A_{1},A_{2},B,\Gamma$ be Borel sets of
$\mathcal{Z},\mathcal{Z},\mathcal{Y},[0,T]$ respectively. Let
$u^{\epsilon}_{i}\in A_{i}, i=1,2$ and let $(\bar{X}^{\epsilon
}(s),\bar{Y}^{\epsilon}(s))$ solve (\ref{Eq:Main2}) with
$u^{\epsilon}_{i}$ in place of $u_{i}$. We associate with
$(\bar{X}^{\epsilon}(s),\bar{Y}^{\epsilon}(s))$ and
$u^{\epsilon}_{i}$ a family of occupation measures
$\mathrm{P}^{\epsilon,\Delta}$ defined by
\begin{equation*}
\mathrm{P}^{\epsilon,\Delta}(A_{1}\times A_{2}\times B\times\Gamma)=\int_{\Gamma}\left[
\frac{1}{\Delta}\int_{t}^{t+\Delta}1_{A_{1}}(u_{1}^{\epsilon}(s))1_{A_{2}}(u_{2}^{\epsilon}(s))1_{B}\left(
\bar{Y}_{s}^{\epsilon}\mod \lambda\right)  ds\right]  dt,
\label{Def:OccupationMeasures2}%
\end{equation*}
We assume that $u_{i}^{\epsilon}(s)=0$ for $i=1,2$ if $s>T$.

For presentation purposes, we devote Subsection \ref{SS:LLN} to the limiting behavior of the controlled process $\left\{\left(\bar{X}^{\epsilon},\bar{Y}^{\epsilon}\right),\epsilon>0\right\}$ in (\ref{Eq:Main2}) as $\epsilon\downarrow 0$. This is a law of large numbers result. The large deviations result is in Subsection \ref{SS:LDP}

\subsection{Limiting
behavior of the controlled process (\ref{Eq:Main2})}\label{SS:LLN}

Theorem \ref{T:MainTheorem1}, deals with the limiting
behavior of the controlled process (\ref{Eq:Main2}) under each of the three
regimes, and uses the notion of a viable pair. 

\begin{theorem}
\label{T:MainTheorem1} Assume Condition \ref{A:Assumption1} and under Regime 1 assume
Condition \ref{A:Assumption2}. Fix the initial point
 $(x_{0},y_{0})\in\mathbb{R}^{m}\times \mathbb{R}^{d-m}$ and consider a family
$\{u^{\epsilon}=(u^{\epsilon}_{1},u^{\epsilon}_{2}),\epsilon>0\}$ of controls in $\mathcal{A}$ satisfying
\begin{equation}
\sup_{\epsilon>0}\mathbb{E}\int_{0}^{T}\left[\left\Vert u_{1}^{\epsilon}(s)\right\Vert
^{2}+\left\Vert u_{2}^{\epsilon}(s)\right\Vert
^{2}\right]ds<\infty\label{Eq:Ubound}%
\end{equation}
Then the family $\{(\bar{X}^{\epsilon
},\mathrm{P}^{\epsilon,\Delta}),\epsilon>0\}$ is tight. Given the particular regime of interaction $i=1,2,3$ and given any subsequence of $\{(\bar{X}^{\epsilon}%
,\mathrm{P}^{\epsilon,\Delta}),\epsilon>0\}$, there exists a
subsubsequence that converges in distribution with limit
$(\bar{X}^{i},\mathrm{P}^{i})$. With probability $1$, the limit
point
$(\bar{X}^{i},\mathrm{P}^{i})\in\mathcal{V}_{(\lambda_{i},\mathcal{L}^{i})}$,
according to Definition \ref{Def:ViablePair}, with the pairs $(\lambda_{i},\mathcal{L}^{i})$ given by Definitions \ref{Def:ThreePossibleOperators} and \ref{Def:ThreePossibleFunctions}.
\end{theorem}

\begin{proof}
We will only present the proof for Regime $1$, since the proof for the other regimes is completely analogous. We start by proving tightness and then we identify the limit.

Tightness of $\{\bar{X}^{\epsilon},\epsilon>0\}$
follows if we establish  that for every $\eta>0$
\begin{equation}
\lim_{\rho\downarrow0}\limsup_{\epsilon\downarrow0}\mathbb{P}_{x_{0},y_{0}}\left[
\sup_{|t_{1}-t_{2}|<\rho,0\leq t_{1}<t_{2}\leq T}|\bar{X}^{\epsilon
}(t_{1})-\bar{X}^{\epsilon}(t_{2})|\geq\eta\right]  =0.\label{Eq:TightnessX}
\end{equation}

The difficulty to obtain this estimate in Regime $1$ is due to the unclear behavior of the term
$\frac{\epsilon}{\delta}\int_{0}^{t}b\left(  \bar{X}^{\epsilon}(s),\bar{Y}^{\epsilon}(s)\right)ds$ as $\epsilon/\delta\uparrow\infty$. We treat this term by applying It\^{o} formula to $\chi(x,y)$, the solution to the cell problem (\ref{Eq:CellProblem}). By doing so, we can rewrite the first component of (\ref{Eq:Main2}), omitting function arguments in some places for notational convenience,  as

\begin{eqnarray}
d\bar{X}^{\epsilon}(s)&=&  \lambda_{1}\left(  \bar{X}^{\epsilon}(s),\bar{Y}^{\epsilon}(s), u_{1}(s),u_{2}(s)\right) ds\nonumber\\
& &\quad +\left(\epsilon \frac{\partial \chi}{\partial x}b+\delta\frac{\partial \chi}{\partial x}\left(c+\sigma u_{1}(s)\right)+\frac{\epsilon\delta}{2}\sigma\sigma^{T}:\frac{\partial^{2} \chi}{\partial x^{2}}+\epsilon\sigma\tau_{1}^{T}:\frac{\partial^{2} \chi}{\partial x\partial y}\right)\left(  \bar{X}^{\epsilon}(s),\bar{Y}^{\epsilon}(s)\right)   ds\nonumber\\
& &\quad+
\left(\sqrt{\epsilon}\left(\sigma+\frac{\partial \chi}{\partial y}\tau_{1}\right)+\sqrt{\epsilon}\delta\frac{\partial \chi}{\partial x}\sigma \right)\left(  \bar{X}^{\epsilon}(s),\bar{Y}^{\epsilon}(s)\right)
dW(s)+\sqrt{\epsilon}\frac{\partial \chi}{\partial y}\tau_{2}\left(  \bar{X}^{\epsilon}(s),\bar{Y}^{\epsilon}(s)\right)dB_{s}\nonumber
\end{eqnarray}

Then, from this representation  and the boundedness of the coefficients and of the derivatives of $\chi$ (Chapter 3, Section 6 of \cite{BLP}), statement (\ref{Eq:TightnessX}) follows, which then gives tightness of $\{\bar{X}^{\epsilon},\epsilon>0\}$.

Tightness of the occupation measures $\{\mathrm{P}%
^{\epsilon,\Delta},\epsilon>0\}$ follows from the bound
\begin{equation}
\sup_{\epsilon\in(0,1]}\mathbb{E}_{x_{0},y_{0}}\left[  g(\mathrm{P}^{\epsilon
,\Delta})\right]  <\infty.\label{Eq:TightnessP}
\end{equation}
for the tightness function $g(r)=\int_{\mathcal{Z}\times\mathcal{Z}\times\mathcal{Y}\times\lbrack0,T]}\left[\left\Vert
z_{1}\right\Vert ^{2}+\left\Vert
z_{2}\right\Vert ^{2}\right]r(dz_{1}dz_{2}dydt),\hspace{0.2cm}r\in\mathcal{P}(\mathcal{Z}^{2}%
\times\mathcal{Y}\times\lbrack0,T])$. To be precise, notice that the function $g(r)$ is a tightness function since it is bounded from below, with relatively compact level sets
$G_{k}=\{r\in\mathcal{P}(\mathcal{Z}^{2}\times\mathcal{Y}\times\lbrack
0,T]):g(r)\leq k\}$ for each $k<\infty$. Then, by Theorem A.19 in \cite{DupuisEllis}, it is known that tightness of $\{\mathrm{P}^{\epsilon,\Delta},\epsilon>0\}$ holds if
(\ref{Eq:TightnessP}) holds. However, this follows directly from Condition (\ref{Eq:Ubound}), due to the estimate
\begin{align}
\sup_{\epsilon\in(0,1]}\mathbb{E}_{x_{0},y_{0}}\left[  g(\mathrm{P}^{\epsilon
,\Delta})\right]   &  =\sup_{\epsilon\in(0,1]}\mathbb{E}_{x_{0},y_{0}}\left[
\int_{\mathcal{Z}^{2}\times\mathcal{Y}\times\lbrack0,T]}\left[\left\Vert z_{1}\right\Vert
^{2}+\left\Vert z_{2}\right\Vert
^{2}\right]\mathrm{P}^{\epsilon,\Delta}(dz_{1}dz_{2}dydt)\right] \nonumber\\
&  =\sup_{\epsilon\in(0,1]}\mathbb{E}_{x_{0},y_{0}}\int_{0}^{T}\frac{1}{\Delta}%
\int_{t}^{t+\Delta}\left[\left\Vert u_{1}^{\epsilon}(s)\right\Vert ^{2}+\left\Vert u_{2}^{\epsilon}(s)\right\Vert ^{2}\right]dsdt\nonumber\\
&  <\infty.\nonumber
\end{align}

Hence, we have established that the family $\{(\bar{X}^{\epsilon},\mathrm{P}^{\epsilon,\Delta
}),\epsilon>0\}$ is tight. Next, we prove that any accumulation point will be a viable pair according to Definition \ref{Def:ViablePair}. Tightness guarantees that for any
subsequence of $\epsilon>0$ there exists subsubsequence that converges, in distribution, to some limit
$(\bar{X},\mathrm{P})$ such that
\[
(\bar{X}^{\epsilon},\mathrm{P}^{\epsilon,\Delta})\rightarrow(\bar
{X},\mathrm{P})
\]
Making use of
the Skorokhod representation theorem  we may assume, by the introduction of another probability space, which we omit writing in the notation, that this
convergence holds with probability $1$, (Theorem 1.8
\cite{EithierKurtz}). Fatou's Lemma gives us
\begin{equation}
\mathbb{E}_{x_{0},y_{0}}\int_{\mathcal{Z}\times\mathcal{Z}\times\mathcal{Y}\times\lbrack
0,T]}\left[\left\Vert z_{1}\right\Vert ^{2}+\left\Vert z_{2}\right\Vert ^{2}\right]\mathrm{P}(dz_{1}dz_{2}dydt)<\infty
\label{A:Assumption2_1}%
\end{equation}
which then implies that $\int_{\mathcal{Z}\times\mathcal{Z}\times\mathcal{Y}\times\lbrack0, T]}\left[\left\Vert
z_{1}\right\Vert ^{2}+\left\Vert
z_{2}\right\Vert ^{2}\right]\mathrm{P}(dz_{1}dz_{2}dydt)<\infty$ \ w.p.1.

Thus, we need now to prove that the limit point $(\bar{X},\mathrm{P})\in\mathcal{V}_{(\lambda_{1},\mathcal{L}^{1})}$,
according to Definition \ref{Def:ViablePair}. Some of the computations here are analogous to those of the proof of Theorem 2.8 in \cite{DupuisSpiliopoulos}. We recall for completeness the main arguments appropriately modified to cover the more general case that is considered here.

We start with (\ref{Eq:AccumulationPointsProcessViable}). This follows from the characterization of solutions to SDE's via the martingale problem formulation and the averaging principle
\cite{BLP, EithierKurtz}.  We fix a collection of elements $p,q,S,t_{i},\tau,F,\phi_{j},\zeta$ that are defined as follows. $S,t_{i},\tau\geq0,i\leq q$ are such that $t_{i}\leq S\leq S+\tau
\leq T$. The real valued functions $F,\phi_{j}$ are smooth and have
compact support. Moreover, $\zeta$ is a real valued, bounded and continuous function with
compact support on $(\mathbb{R}^{m})^{q}\times\mathbb{R}^{pq}$.

Then, we define $\bar{\mathcal{A}}_{t}^{\epsilon,\Delta}$ by
\begin{equation}
\bar{\mathcal{A}}_{t}^{\epsilon,\Delta}F(x)=\int_{\mathcal{Z}\times\mathcal{Z}\times
\mathcal{Y}}\lambda(x,y,z_{1},z_{2})\nabla F(x)\mathrm{P}_{t}^{\epsilon,\Delta}(dz_{1}dz_{2}dy)
\label{Eq:PrelimitOperator}%
\end{equation}
where
\[
\mathrm{P}_{t}^{\epsilon,\Delta}(dz_{1}dz_{2}dy)=\frac{1}{\Delta}\int_{t}^{t+\Delta
}1_{dz_{1}}(u_{1}^{\epsilon}(s))1_{dz_{2}}(u_{2}^{\epsilon}(s))1_{dy}\left(  \bar{Y}^{\epsilon}(s)\mod \lambda\right)  ds.
\]
With these definitions at hand,  (\ref{Eq:AccumulationPointsProcessViable}) follow, if we prove that, as $\epsilon\downarrow0$,
\begin{equation}
\mathbb{E}_{x_{0},y_{0}}\left[  \zeta(\bar{X}_{t_{i}}^{\epsilon},(\mathrm{P}%
^{\epsilon,\Delta},\phi_{j})_{t_{i}},i\leq q,j\leq p)\left[  F(\bar{X}^{\epsilon}(S+\tau))-F(\bar{X}^{\epsilon}(S))-\int_{S}^{S+\tau}%
\bar{\mathcal{A}}_{t}^{\epsilon,\Delta}F(\bar{X}^{\epsilon}(t))dt\right]
\right]  \rightarrow0 \label{Eq:MartingaleProblemRegime1_1}%
\end{equation}
and, in probability,
\begin{equation}
\int_{S}^{S+\tau}\bar{\mathcal{A}}_{s}^{\epsilon,\Delta}F(\bar{X}^{\epsilon}(s))ds-\int_{\mathcal{Z}\times\mathcal{Z}\times\mathcal{Y}\times\lbrack S,S+\tau
]}\lambda_{1}(\bar{X}_{s},y,z_{1},z_{2})\nabla F(\bar{X}(s))\bar{\mathrm{P}}(dz_{1}dz_{2}dyds)\rightarrow0.
\label{Eq:MartingaleProblemRegime1_2a}%
\end{equation}
Then,  the pair $(\bar{X},\bar{\mathrm{P}})$ solves the martingale problem associated with (\ref{Eq:AccumulationPointsProcessViable}), which implies the latter.

So, let us first prove (\ref{Eq:MartingaleProblemRegime1_1}). Notice that under the topology of weak convergence and due to the fact that the last marginal of
$\mathrm{P}$ is Lebesgue measure w.p.1., we have that for every  $t\in\lbrack0,T]$ and for every real valued, compactly supported and continuous function $\phi$
\[
(\mathrm{P}^{\epsilon,\Delta},\phi)_{t}\rightarrow(\mathrm{P},\phi)_{t}\text{
w.p.}1.
\]

Next, we recall the solution
$\chi(x,y)$ to the cell problem (\ref{Eq:CellProblem}).  Let $\psi=\{\psi_{1},\ldots,\psi_{d}\}$ be defined by $\psi_{\ell}(x,y)=\chi_{\ell}(x,y)F_{x_{\ell}}(x)$ for $\ell
=1,\ldots,d$. Notice that $\psi_{\ell}(x,y)$ is periodic in every direction in $y$, with period $\lambda$, and satisfies
\begin{equation}
\mathcal{L}_{x}^{1}\psi_{\ell}(x,y)=-b_{\ell}(x,y)F_{x_{\ell}}(x),\quad
\int_{\mathcal{Y}}\psi_{\ell}(x,y)\mu(dy|x)=0. \label{Eq:CellProblem2}%
\end{equation}

Applying It\^{o}'s formula to
$\psi(\bar{X}^{\epsilon}(s),\bar{Y}^{\epsilon}(s))$ gives us that relation
(\ref{Eq:CellProblem2}) and the boundedness of $\chi(x,y)$ and its derivatives
 guarantee the validity of
(\ref{Eq:MartingaleProblemRegime1_1}), if
\begin{equation}
\int_{S}^{S+\tau}\left[  \bar{\mathcal{A}}_{s}^{\epsilon,\Delta}F(\bar{X}^{\epsilon}(s))-\lambda_{1}\left(  \bar{X}^{\epsilon}(s),\bar{Y}^{\epsilon}(s),u_{1}^{\epsilon}(s),u_{2}^{\epsilon}(s)\right)  \nabla F(\bar{X}^{\epsilon}(s))\right]  ds\rightarrow0,\text{ as }\epsilon\downarrow0.
\label{Eq:MartingaleProblemRegime1_2}%
\end{equation}
in probability. However, this is exactly the second statement of Lemma 3.2 in \cite{DupuisSpiliopoulos} by taking as $g(x,y,z_{1},z_{2})=\lambda_{1}(x,y,z_{1},z_{2})\cdot\nabla f(x)$. On the other hand, (\ref{Eq:MartingaleProblemRegime1_2a}) is the first statement of Lemma 3.2 in \cite{DupuisSpiliopoulos}, with the function $g$ as was just specified. These give the proof of
(\ref{Eq:AccumulationPointsProcessViable}).

Next, we establish relation  (\ref{Eq:AccumulationPointsMeasureViable}). Let $\mathcal{A}^{\epsilon}_{z_{1},z_{2},x}$ be the operator associated
with the fast motion $\bar{Y}^{\epsilon}$ in (\ref{Eq:Main2}) with $z_{1}=u_{1}$, $z_{2}=u_{2}$ and $x=\bar{X}^{\epsilon}$ fixed,
 \begin{eqnarray}
\mathcal{A}_{z_{1},z_{2},x}^{\epsilon}F(y)&=&\left[  \frac{\epsilon}{\delta^{2}%
}f(x,y)+\frac{1}{\delta}\left[  g(x,y)+\tau_{1}(x,y)z_{1}+\tau_{2}(x,y)z_{2}\right]  \right]
\cdot\nabla_{y}F(y)+\nonumber\\
& &\hspace{6cm}+\frac{\epsilon}{\delta^{2}}\frac{1}{2}\left(\tau_{1}\tau_{1}^{T}+\tau_{2}\tau_{2}^{T}\right)(x,y):\nabla_{y}%
\nabla_{y}F(y) \label{Def_of_A_gen}%
\end{eqnarray}
for functions $F\in\mathcal{C}^{2}(\mathcal{Y})$. Consider,  $\left\{F_{\ell}:\mathcal{Y}\mapsto\mathbb{R}, \ell\in\mathbb{N}\right\}$ to be a smooth and dense
family in $\mathcal{C}^{2}(\mathcal{Y})$. Then, notice that
\begin{equation*}
 M^{\epsilon}_{t}=F_{\ell}(Y^{\epsilon}(t))-F_{\ell}(y_{0})-\int_{0}^{t}\mathcal{A}^{\epsilon}_{u_{1}(s),u_{2}(s),\bar{X}^{\epsilon}(s)}F_{\ell}(\bar{Y}^{\epsilon}(s))ds
\end{equation*}
is an $\mathfrak{F}_{t}$-martingale. Next, let
\begin{equation*}
g(\epsilon)=\frac{\delta^{2}}{\epsilon}
\end{equation*}
and notice that $g(\epsilon
)\mathcal{A}_{z_{1},z_{2},x}^{\epsilon}$ converges to $\mathcal{L}_{z_{1},z_{2},x}^{1}$ under
Regime $1$, as
$\epsilon\downarrow0$. Based on this observation and denoting $\mathcal{G}_{x,y,z_{1},z_{2}}F_{\ell}(y)=\left[  g(x,y)+\tau_{1}(x,y)z_{1}+\tau_{1}(x,y)z_{2}\right]  \cdot
\nabla_{y}F_{\ell}(y)$
 we write

\begin{align}
& g(\epsilon)M_{t}^{\epsilon}-g(\epsilon)\left[  F_{\ell}(\bar{Y}
^{\epsilon}(t))-F_{\ell}(y_{0})\right] \nonumber\\
&\quad \mbox{} +g(\epsilon
)\left[\int_{0}^{t}\frac
{1}{\Delta}\left[  \int_{s}^{s+\Delta}\mathcal{A}_{u^{\epsilon}_{1}(\rho),u^{\epsilon}_{2}(\rho),\bar
{X}^{\epsilon}(\rho)}^{\epsilon}F_{\ell}(\bar{Y}^{\epsilon}(\rho))d\rho\right]  ds-\int_{0}^{t}\mathcal{A}_{u^{\epsilon}_{1}(s),u^{\epsilon}_{2}(s),\bar{X}^{\epsilon}(s)}^{\epsilon
}F_{\ell}(\bar{Y}^{\epsilon}(s))ds\right]\nonumber\\
&  =-\frac{\delta}{\epsilon}\left(  \int_{0}^{t}\frac{1}{\Delta}\left[
\int_{s}^{s+\Delta}\left[  \mathcal{G}_{\bar{X}^{\epsilon}(\rho),\bar{Y}%
^{\epsilon}(\rho),u^{\epsilon}_{1}(\rho),u^{\epsilon}_{2}(\rho)}F_{\ell}(\bar{Y}^{\epsilon}(\rho)
)-\mathcal{G}_{\bar{X}^{\epsilon}(s),\bar{Y}^{\epsilon}(\rho),u^{\epsilon}_{1}(\rho),u^{\epsilon}_{2}(\rho)
}F_{\ell}(\bar{Y}^{\epsilon}(\rho))\right]  d\rho\right]  ds\right) \nonumber\\
& \quad \mbox{}-\frac{\delta}{\epsilon}\left(  \int_{\mathcal{Z}\times\mathcal{Z}\times
\mathcal{Y}\times\lbrack0,t]}\mathcal{G}_{\bar{X}^{\epsilon}(s),y,z_{1},z_{2}}F_{\ell
}(y)\bar{\mathrm{P}}^{\epsilon,\Delta}(dz_{1}dz_{2}dyds)\right) \nonumber\\
& \quad \mbox{}-\int_{0}^{t}\frac{1}{\Delta}\left[  \int_{s}^{s+\Delta}\left[
\mathcal{L}_{\bar{X}^{\epsilon}{\rho}}^{1}F_{\ell}(\bar{Y}^{\epsilon
}(\rho))-\mathcal{L}_{\bar{X}^{\epsilon}(s)}^{1}F_{\ell}(\bar{Y}^{\epsilon
}(\rho))\right]  d\rho\right]  ds\nonumber\\
&\quad  \mbox{}-\int_{\mathcal{Z}\times\mathcal{Z}\times\mathcal{Y}\times\lbrack0,t]}\mathcal{L}%
_{\bar{X}^{\epsilon}(s)}^{1}F_{\ell}(y)\mathrm{P}^{\epsilon,\Delta}(dz_{1}dz_{2}dydt).
\label{Eq:MartingaleProperty_1}%
\end{align}

Let us now analyze the different terms in (\ref{Eq:MartingaleProperty_1}). In particular
\begin{enumerate}
\item{We have that $g(\epsilon)M^{\epsilon}_{t}\downarrow 0$ as $\epsilon\downarrow 0$ in probability.  Indeed, we can rewrite
 \begin{equation*}
 M_{t}^{\epsilon}=\frac{\sqrt{\epsilon}}{\delta}\left[\int_{0}%
^{t}\nabla_{y}F_{\ell}(\bar{Y}^{\epsilon}(s))\cdot\tau_{1}\left(  \bar{X}^{\epsilon}(s),\bar{Y}^{\epsilon}(s)\right)  dW(s)+\int_{0}%
^{t}\nabla_{y}F_{\ell}(\bar{Y}_{s}^{\epsilon})\cdot\tau_{2}\left(  \bar{X}^{\epsilon}(s),\bar{Y}^{\epsilon}(s)\right)  dB(s)\right], \label{Def:Martingale}%
 \end{equation*}
which allows us to obtain that
$\mathbb{E}_{x_{0},y_{0}}\left[
M_{T}^{\epsilon}\right]  ^{2}\leq C_{0}\frac{1}{g(\epsilon)}$, and so $g(\epsilon)M^{\epsilon}_{t}\downarrow 0$ follows
 from $g(\epsilon)\downarrow0$.}
 \item{Since
$F$ is bounded, we have that  $g(\epsilon)\left[  F_{\ell}(\bar{Y}^{\epsilon
}(t))-F_{\ell}(y_{0})\right]  $ converges to zero uniformly.}
\item{By Condition
\ref{A:Assumption1} and since $\Delta\downarrow
0,\delta/\epsilon\downarrow0$, the term
\[
g(\epsilon
)\left[\int_{0}^{t}\frac
{1}{\Delta}\left[  \int_{s}^{s+\Delta}\mathcal{A}_{u^{\epsilon}_{1}(\rho),u^{\epsilon}_{2}(\rho),\bar
{X}^{\epsilon}(\rho)}^{\epsilon}F_{\ell}(\bar{Y}^{\epsilon}(\rho))d\rho\right]  ds-\int_{0}^{t}\mathcal{A}_{u^{\epsilon}_{1}(s),u^{\epsilon}_{2}(s),\bar{X}^{\epsilon}(s)}^{\epsilon
}F_{\ell}(\bar{Y}^{\epsilon}(s))ds\right]
\]
converges to zero in probability.
}
\item{Tightness of $\left\{\bar{X}^{\epsilon},\epsilon>0\right\}$ and Conditions \ref{A:Assumption1},
(\ref{Eq:Ubound}), imply that the
first and the third term in the right hand side of
(\ref{Eq:MartingaleProperty_1}) converge to zero in probability as $\delta/\epsilon\downarrow0$.}
\item{Uniform integrability of
$\mathrm{P}^{\epsilon,\Delta}$ and the fact that $\delta/\epsilon
\downarrow0$ imply that the
second term on the right hand side of (\ref{Eq:MartingaleProperty_1})
converges to zero in probability.}
\end{enumerate}

Therefore, we finally obtain that

\[\int_{\mathcal{Z}\times\mathcal{Z}\times\mathcal{Y}\times\lbrack0,T]}\mathcal{L}%
_{z_{1},z_{2},\bar{X}_{t}^{\epsilon}}^{i}F_{\ell}(y)\mathrm{P}^{\epsilon,\Delta}(dz_{1}dz_{2}dydt)\rightarrow 0, \quad \textrm{ in probability.}
\]
Then, this implies (\ref{Eq:AccumulationPointsMeasureViable}), by continuity in $t\in[0,T]$ and density of
 $\left\{F_{\ell}:\mathcal{Y}\mapsto\mathbb{R}, \ell\in\mathbb{N}\right\}$.

It remains to prove (\ref{Eq:AccumulationPointsFullMeasureViable}). It is clear, that the analogous
property holds at the prelimit level. Moreover, since  $\mathrm{P}(\mathcal{Z}\times
\mathcal{Z}\times\mathcal{Y}\times\left\{  t\right\}  )=0$ and the map
$t\rightarrow\mathrm{P}(\mathcal{Z}\times\mathcal{Z}\times\mathcal{Y}\times\lbrack0,t])$ is continuous, one can deal with null sets. Thus, (\ref{Eq:AccumulationPointsFullMeasureViable}) follows.

The proof for Regimes $2$ and $3$ is completely analogous with the only exception that for the purposes of the proof (\ref{Eq:AccumulationPointsMeasureViable}), for Regime $2$ we define $g(\epsilon)=\epsilon$ and for Regime $3$, $g(\epsilon)=\delta$. This completes the proof of the theorem.

\end{proof}

\subsection{Large deviations for $\{X^{\epsilon},\epsilon>0\}$}\label{SS:LDP}
In this subsection we present the main large deviations result. The main difference from the case considered in \cite{DupuisSpiliopoulos} is the identification
of the correct viable pair with respect to which the large
deviations principle is expressed to. The proper viable pair in each regime is indicated by Theorem \ref{T:MainTheorem1}.

\begin{theorem}
\label{T:MainTheorem2} Let $\{\left(X^{\epsilon},Y^{\epsilon}\right),\epsilon>0\}$ be the
unique strong solution to (\ref{Eq:Main}) and consider Regime $i=1,2,3$. Assume Condition
\ref{A:Assumption1} and under Regime 1 assume Condition \ref{A:Assumption2}. Define
\begin{equation}
S^{i}(\phi)=\inf_{(\phi,\mathrm{P})\in\mathcal{V}_{(\lambda_{i},\mathcal{L}%
^{i})}}\left[  \frac{1}{2}\int_{\mathcal{Z}\times\mathcal{Z}\times\mathcal{Y}\times\lbrack
0,T]}\left[\left\Vert z_{1}\right\Vert ^{2}+\left\Vert z_{2}\right\Vert ^{2}\right]\mathrm{P}(dz_{1}dz_{2}dydt)\right]  ,
\label{Eq:GeneralRateFunction}%
\end{equation}
with the convention that the infimum over the empty set is $\infty$. The pairs $(\lambda_{i},\mathcal{L}^{i})$ are given in Definitions \ref{Def:ThreePossibleOperators} and \ref{Def:ThreePossibleFunctions}. Then, we have
\begin{enumerate}
 \item The level sets of $S^{i}$ are compact. In particular, for each $s<\infty$, the set
\begin{equation*}
\Phi_{s}^{i}=\{\phi\in\mathcal{C}([0,T];\mathbb{R}^{m}):S^{i}(\phi)\leq s\}
\label{Def:LevelSets}%
\end{equation*}
is a compact subset of $\mathcal{C}([0,T];\mathbb{R}^{m})$.
\item For
every bounded and continuous function $h$ mapping $\mathcal{C}%
([0,T];\mathbb{R}^{m})$ into $\mathbb{R}$
\begin{equation*}
\liminf_{\epsilon\downarrow0}-\epsilon\ln\mathbb{E}_{x_{0},y_{0}}\left[  \exp\left\{
-\frac{h(X^{\epsilon})}{\epsilon}\right\}  \right]  \geq \inf_{\phi\in
\mathcal{C}([0,T];\mathbb{R}^{m})}\left[  S^{i}(\phi)+h(\phi)\right]  .
\end{equation*}
\item In the case of
Regime $3$ assume either that we are in dimension $1$ i.e.,$m=1,d=2$, or that
$g(x,y)=g(y)$ and $\tau_{i}(x,y)=\tau_{i}(y), i=1,2$ for the general
multidimensional case.  Then for
every bounded and continuous function $h$ mapping $\mathcal{C}%
([0,T];\mathbb{R}^{m})$ into $\mathbb{R}$
\begin{equation*}
\limsup_{\epsilon\downarrow0}-\epsilon\ln\mathbb{E}_{x_{0},y_{0}}\left[  \exp\left\{
-\frac{h(X^{\epsilon})}{\epsilon}\right\}  \right]  \leq \inf_{\phi\in
\mathcal{C}([0,T];\mathbb{R}^{m})}\left[  S^{i}(\phi)+h(\phi)\right]  .
\end{equation*}
\end{enumerate}
In other words, under the imposed assumptions,
$\{X^{\epsilon},\epsilon>0\}$ satisfies the large deviations principle with action functional $S^{i}$.
\end{theorem}

For the sake of presentation, the proof of Theorem \ref{T:MainTheorem2} is deferred to the end of this section. In the case of Regime $1$ we can get an explicit characterization of the rate function.

\begin{theorem}
\label{T:MainTheorem3} Let $\{\left(X^{\epsilon},Y^{\epsilon}\right),\epsilon>0\}$ be the unique strong
solution to (\ref{Eq:Main}) and consider Regime $1$. Under Conditions
\ref{A:Assumption1} and \ref{A:Assumption2}, $\{X^{\epsilon},\epsilon>0\}$
satisfies a large deviations principle with rate function
\begin{equation*}
S(\phi)=%
\begin{cases}
\frac{1}{2}\int_{0}^{T}(\dot{\phi}(s)-r(\phi(s)))^{T}q^{-1}(\phi(s)%
)(\dot{\phi}(s)-r(\phi(s)))ds & \text{if }\phi\in\mathcal{AC}%
([0,T];\mathbb{R}^{m}) \text{  and } \phi(0)=x_{0}\\
+\infty & \text{otherwise.}%
\end{cases}
\label{ActionFunctional1}%
\end{equation*}

\end{theorem}
\begin{proof}
It follows by putting Lemma \ref{T:MainTheorem21} and Theorem \ref{T:ExplicitSolutionRegime1} below together.
\end{proof}

Notice that the coefficients $r(x)$ and $q(x)$ that enter into the
action functional for Regime $1$ are those obtained if we had first
taken to (\ref{Eq:Main}) $\delta\downarrow 0$ with $\epsilon$ fixed
and then consider the large deviations for the homogenized system.
Indeed if $\epsilon=1$, then $X^{\epsilon,\delta}=X^{1,\delta}$ can be shown to converge
weakly in the space of continuous functions in
$\mathcal{C}([0,T];\mathbb{R}^{m})$, as $\delta\downarrow 0$, to the solution of an SDE with
drift coefficient $r(x)$ and diffusion coefficient $q^{1/2}(x)$.
This can be derived via standard homogenization theory
\cite{BLP,PS}. The action functional for a small noise diffusion with drift coefficient $r(x)$ and diffusion coefficient $\sqrt{\epsilon}q^{1/2}(x)$ is the one given by Theorem \ref{T:MainTheorem3}. This is in accordance to intuition since under Regime
$1$, $\delta$ goes to zero faster, so homogenization should occur
first as it indeed does.

\begin{remark}
Notice that if we set $f=b,g=c,\sigma=\tau_{1}$ and $\tau_{2}=0$ in
the statements of Theorems
\ref{T:MainTheorem1}-\ref{T:MainTheorem3}, then one recovers the
results of \cite{DupuisSpiliopoulos}.
\end{remark}

The reader may wonder, why we have imposed further structural restrictions for part (iii) of the theorem for Regime $3$. This is because we were not able to prove some smoothness
requirements of the constructed nearly optimal controls in the prelimit level with respect to  $x$ in the general multidimensional case when the coefficients $g,\tau_{1}$ and $\tau_{2}$ depend on $x$.
Similar issues arise in the case considered in \cite{DupuisSpiliopoulos} and  are discussed in detail there. However, observing the viable pairs that characterize the large deviations
principle for Regimes $2$ and $3$, $(\lambda_{2},\mathcal{L}^{2})$
and $(\lambda_{3},\mathcal{L}^{3})$  respectively, we notice that
Regime $3$ can be thought of as a limiting case of Regime $2$ with
$\gamma=0$. So, one is led to conjecture that the extra assumptions for Regime $3$ are not necessary, even though we currently do not have a proof for this.

Next, we proceed with the proof of Theorem \ref{T:MainTheorem2}. The proof is based on two intermediate results, which we present now. Lemma \ref{T:MainTheorem21} allows us to equivalently  rewrite the action functional $S^{i}$.

Let $\mathcal{AC}([0,T];\mathbb{R}^{m})$ be the set of absolutely continuous functions from $[0,T]$ to $\mathbb{R}^{m}$.

\begin{lemma}
\label{T:MainTheorem21} Consider the set-up of Theorem \ref{T:MainTheorem2}. Then, for $i=1,2,3$ we have
\begin{align}
S^{i}(\phi)=
\begin{cases}
\int_{0}^{T}L_{i}(\phi(s),\dot{\phi}(s))ds & \text{if }\phi
\in\mathcal{AC}([0,T];\mathbb{R}^{m}) \text{  and } \phi(0)=x_{0}\\
+\infty & \text{otherwise }.
\end{cases}
\nonumber
\end{align}
where
\begin{equation}
L_{i}(x,\beta)=\inf_{(v,\mu)\in\mathcal{A}_{x,\beta}^{i}}\left\{
\frac{1}{2}\int_{\mathcal{Y}}\left\Vert v(y)\right\Vert ^{2}\mu(dy)\right\}.
\label{Def:LocalRateFunction}%
\end{equation}
with
\begin{align*}
\mathcal{A}_{x,\beta}^{i}  &  =\left\{
v(\cdot)=(v_{1}(\cdot),v_{2}(\cdot)):\mathcal{Y}\mapsto
\mathbb{R}^{2\kappa},\mu\in\mathcal{P}(\mathcal{Y})\hspace{0.1cm}:\hspace
{0.1cm}(v,\mu)\text{ satisfy }\int_{\mathcal{Y}}\mathcal{L}_{v_{1}(y),v_{2}(y),x}%
^{i}F(y)\mu(dy)=0\right. \\
&  \left.  \hspace{1.8cm}\text{ for all }F\in \mathcal{D}\left(\mathcal{L}_{v_{1},v_{2},x}^{i} \right),
\int_{\mathcal{Y}}\left\Vert v(y)\right\Vert ^{2}\mu(dy)<\infty\text{ and
}\beta=\int_{\mathcal{Y}}\lambda_{i}(x,y,v_{1}(y),v_{2}(y))\mu(dy)\right\}.
\end{align*}
\end{lemma}
\begin{proof}
The proof of this lemma follows easily by an appropriate rewriting of the corresponding expressions. First, notice that (\ref{Eq:GeneralRateFunction}) can be written in terms of a local rate function
\begin{equation*}
S^{i}(\phi)=\int_{0}^{T}L_{i}^{r}(\phi(s),\dot{\phi}(s))ds,
\label{Eq:GeneralRateFncRegime1}
\end{equation*}
(if $\phi$ is absolutely continuous). This follows from the definition of a viable pair by setting
\begin{equation}
L_{i}^{r}(x,\beta)=\inf_{\mathrm{P}\in\mathcal{A}_{x,\beta}^{i,r}}%
\int_{\mathcal{Z}\times\mathcal{Z}\times\mathcal{Y}}\frac{1}{2}\left[\left\Vert z_{1}\right\Vert
^{2}+\left\Vert z_{2}\right\Vert
^{2}\right]\mathrm{P}(dz_{1}dz_{2}dy), \label{Def:ErgodicControlConstraints1_1Relaxed}%
\end{equation}
where
\begin{align*}
\mathcal{A}_{x,\beta}^{i,r}  &  =\left\{  \mathrm{P}\in\mathcal{P}%
(\mathcal{Z}\times\mathcal{Z}\times\mathcal{Y}):\int_{\mathcal{Z}\times\mathcal{Z}\times\mathcal{Y}}%
\mathcal{L}_{z_{1},z_{2},x}^{i}F(y)\mathrm{P}(dz_{1}dz_{2}dy)=0\text{ for all }F\in \mathcal{D}\left(\mathcal{L}_{z_{1},z_{2},x}^{i} \right)\right. \\
&  \left.  \hspace{2cm}\int_{\mathcal{Z}\times\mathcal{Z}\times\mathcal{Y}}\left[\left\Vert
z_{1}\right\Vert ^{2}+\left\Vert
z_{2}\right\Vert ^{2}\right]\mathrm{P}(dz_{1}dz_{2}dy)<\infty\text{ and }\beta=\int_{\mathcal{Z}%
\times\mathcal{Z}\times\mathcal{Y}}\lambda_{i}(x,y,z_{1},z_{2})\mathrm{P}(dz_{1}dz_{2}dy)\right\}  .
\end{align*}
Second, we note that $\mathrm{P}\in\mathcal{P}(\mathcal{Z}\times\mathcal{Z}\times\mathcal{Y})$ can be decomposed into marginals as follows
\begin{equation*}
\mathrm{P}(dz_{1}dz_{2}dy)=\eta(dz_{1}dz_{2}|y)\mu(dy).
\end{equation*}
This, the convexity of the cost on $(z_{1},z_{2})$ and the affine dependence of $\lambda_{i}$ on $(z_{1},z_{2})$ imply that the relaxed control formulation
(\ref{Def:ErgodicControlConstraints1_1Relaxed}) and the ordinary control formulation (\ref{Def:LocalRateFunction}) are equivalent by taking
\begin{equation*}
v_{i}(y)=\int_{\mathcal{Z}\times\mathcal{Z}}z_{i}\eta(dz_{1}dz_{2}|y).
\end{equation*}
Hence, the statement of the lemma holds.
\end{proof}

Now, we use the representations in Lemma \ref{T:MainTheorem21} to obtain the controls needed in the proof of part (iii) of Theorem \ref{T:MainTheorem2}. In the case of  Regime 1 we can be even more specific and obtain a closed form expression for the variational problem associated to the local rate
function $L_{1}(x,\beta)$ appearing in Lemma \ref{T:MainTheorem21}. The derivation of the closed form expression is based on identifying an optimal control that is then used to prove Theorem \ref{T:MainTheorem2}.
The proof of this statement is based on a straightforward Lagrange multiplier type of analysis of the variational problem (\ref{Def:LocalRateFunction}) for $i=1$ and thus omitted
 (see also Theorem 5.2 in \cite{DupuisSpiliopoulos} for an analogous situation). In the case of Regime 2 and Regime 3, we can obtain that there is pair
$(\bar{v},\mu)$ that attains the infimum in (\ref{Def:LocalRateFunction}).
We collect these statements in the following theorem. Its proof is omitted, since it follows analogously to the corresponding proofs of Theorems 5.2, 6.2 for Regimes 1 and 2 respectively and from Section 7 for Regime 3, of \cite{DupuisSpiliopoulos}.

\begin{theorem}
\label{T:ExplicitSolutionRegime1} Assume Condition \ref{A:Assumption1} and in the case of Regime $1$ assume Condition
\ref{A:Assumption2}. The infimization problem
(\ref{Def:LocalRateFunction}) for $i=1$ has the explicit solution%
\[
L_{1}(x,\beta)=\frac{1}{2}(\beta-r(x))^{T}q^{-1}(x)(\beta-r(x)),
\]
where

\begin{itemize}
\item {$r(x)=\int_{\mathcal{Y}}\left(c(x,y)+\frac{\partial\chi}{\partial y}(x,y)g(x,y)\right)
\mu(dy|x)$,}
\item {$q(x)=\int_{\mathcal{Y}}\left[\left(\sigma+\frac{\partial\chi}{\partial y}\tau_{1}\right)\left(\sigma+\frac{\partial\chi}{\partial y}\tau_{1}\right)^{T}+\left(\frac{\partial\chi}{\partial y}\tau_{2}\right)\left(\frac{\partial\chi}{\partial y}\tau_{2}\right)^{T}\right](x,y)\mu(dy|x),$}
\end{itemize}

and where $\mu(dy|x)$ is the unique invariant measure corresponding to the
operator $\mathcal{L}_{x}^{1}$ and $\chi(x,y)$ is defined by
(\ref{Eq:CellProblem}). The control
\begin{equation*}
\bar{u}(y)=\left(\bar{u}_{1,\beta}(x,y),\bar{u}_{2,\beta}(x,y)\right)=\left(\left(\sigma+\frac{\partial\chi}{\partial y}\tau_{1}\right)^{T}(x,y)q^{-1}(x)(\beta-r(x)), \left(\frac{\partial\chi}{\partial y}\tau_{2}\right)^{T}(x,y)q^{-1}(x)(\beta-r(x))\right)
\end{equation*}
attains the infimum in (\ref{Def:LocalRateFunction}).

In the case of Regime 2 and in the one dimensional setting of Regime 3, there  is a pair $(\bar{u},\bar{\mu})$ that achieves the infimum in (\ref{Def:LocalRateFunction}) such that $\bar{u}=\bar{u}_{\beta}(x,y)$
is, for each fixed $\beta\in\mathbb{R}^{d}$, continuous in $x$, Lipschitz
continuous in $y$ and measurable in $(x,y,\beta)$. Moreover, $\bar{\mu
}(dy)=\bar{\mu}_{\bar{u}}(dy|x)$ is the unique invariant measure corresponding
to the operator $\mathcal{L}_{\bar{u}_{\beta}(x,y),x}^{i}$, $i=2,3$, and it is weakly
continuous as a function of $x$. In the $x-$independent multidimensional case of Regime $3$, there is a $\mathrm{P}\in\mathcal{A}_{\beta}^{3,r}$ that achieves the infimum in (\ref{Def:ErgodicControlConstraints1_1Relaxed}) or equivalently in (\ref{Def:LocalRateFunction}).
\end{theorem}

Now we have the necessary tools to prove Theorem \ref{T:MainTheorem2}.

\begin{proof}[Proof of Theorem \ref{T:MainTheorem2}]
Part (i). As in Lemma 4.1 and Lemma 4.3 of \cite{DupuisSpiliopoulos}, one can establish that $\Phi^{i}_{s}$ is precompact and closed, respectively. These two statements, then give
compactness of $\Phi_{s}^{i}$.

Part (ii). We have the following chain of inequalities.
\begin{align*}
&\liminf_{\epsilon\downarrow0}\left(  -\epsilon\ln\mathbb{E}_{x_{0}, y_{0}}\left[
\exp\left\{  -\frac{h(X^{\epsilon})}{\epsilon}\right\}  \right]  \right)
\geq\liminf_{\epsilon\downarrow0}\left(  \mathbb{E}_{x_{0}, y_{0}}\left[  \frac{1}%
{2}\int_{0}^{T}\left[\left\Vert u_{1}^{\epsilon}(t)\right\Vert ^{2}+\left\Vert u_{2}^{\epsilon}(t)\right\Vert ^{2}\right]dt+h(\bar
{X}^{\epsilon})\right]  -\epsilon\right) \nonumber\\
& \qquad \geq\liminf_{\epsilon\downarrow0}\left(  \mathbb{E}_{x_{0},y_{0}}\left[  \frac
{1}{2}\int_{0}^{T}\frac{1}{\Delta}\int_{t}^{t+\Delta}\left[\left\Vert u_{1}
^{\epsilon}(s)\right\Vert ^{2}+\left\Vert u_{2}
^{\epsilon}(s)\right\Vert ^{2}\right]dsdt+h(\bar{X}^{\epsilon})\right]  \right)
\nonumber\\
& \qquad =\liminf_{\epsilon\downarrow0}\left(  \mathbb{E}_{x_{0}, y_{0}}\left[  \frac{1}%
{2}\int_{\mathcal{Z}\times\mathcal{Y}\times\lbrack0,T]}\left[\left\Vert z_{1}\right\Vert
^{2}+\left\Vert z_{2}\right\Vert
^{2}\right]\mathrm{P}^{\epsilon,\Delta}(dz_{1}dz_{2}dydt)+h(\bar{X}^{\epsilon})\right]  \right)
\nonumber\\
& \qquad \geq\mathbb{E}_{x_{0},y_{0}}\left[  \frac{1}{2}\int_{\mathcal{Z}\times
\mathcal{Y}\times\lbrack0,T]}\left[\left\Vert z_{1}\right\Vert
^{2}+\left\Vert z_{2}\right\Vert
^{2}\right]\bar{\mathrm{P}}%
(dz_{1}dz_{2}dydt)+h(\bar{X})\right] \nonumber\\
& \qquad \geq\inf_{(\phi,\mathrm{P})\in\mathcal{V}}\left\{  \frac{1}{2}%
\int_{\mathcal{Z}\times\mathcal{Z}\times\mathcal{Y}\times\lbrack0,T]}\left[\left\Vert z_{1}\right\Vert
^{2}+\left\Vert z_{2}\right\Vert
^{2}\right]\mathrm{P}(dz_{1}dz_{2}dydt)+h(\phi)\right\}  \nonumber\\
& \qquad =\inf_{\phi\in
\mathcal{C}([0,T];\mathbb{R}^{m})}\left[  S^{i}(\phi)+h(\phi)\right].
\label{Eq:LaplacePrincipleLowerBoundGeneral}%
\end{align*}
The first line follows by the representation formula (\ref{Eq:VariationalRepresentation}). The fourth line by Theorem \ref{T:MainTheorem1} and Fatou's Lemma. This establishes the lower bound.

Part (iii). In each regime we follow the same general steps.  What differs from regime to regime, is the form of the viable pair
$\left(\lambda_{i},\mathcal{L}^{i}\right)$ in the definition of the action functional $S^{i}(\cdot)$.   To
prove the Laplace principle upper bound we must show that for all bounded,
continuous functions $h$ mapping $\mathcal{C}([0,T];\mathbb{R}^{m})$ into
$\mathbb{R}$
\begin{equation*}
\limsup_{\epsilon\downarrow0}-\epsilon\ln\mathbb{E}_{x_{0}, y_{0}}\left[
\exp\left\{  -\frac{h(X^{\epsilon})}{\epsilon}\right\}  \right]  \leq
\inf_{\phi\in\mathcal{C}([0,T];\mathbb{R}^{m})}\left[  S^{i}(\phi)+h(\phi)\right]. 
\end{equation*}
By the variational representation formula (\ref{Eq:VariationalRepresentation}), it is enough to prove that

\begin{equation}
\limsup_{\epsilon\downarrow0}\inf_{u=(u_{1},u_{2})\in\mathcal{A}}\mathbb{E}_{x_{0},y_{0}%
}\left[  \frac{1}{2}\int_{0}^{T}\left[\left\Vert u_{1}(s)\right\Vert ^{2}+\left\Vert u_{2}(s)\right\Vert ^{2}\right]ds+h(\bar
{X}^{\epsilon})\right]  \leq
\inf_{\phi\in\mathcal{C}([0,T];\mathbb{R}^{m})}\left[  S^{i}(\phi)+h(\phi)\right]. \label{Eq:LaplacePrincipleUpperBoundRegime1}%
\end{equation}

In each regime, we consider for the limiting variational problem in the Laplace
principle a nearly optimal control pair $(\psi,\mathrm{P})$. In particular, let $\eta>0$ be given and consider $\psi\in\mathcal{C}([0,T];\mathbb{R}^{m})$
with $\psi_{0}=x_{0}$ such that
\begin{equation*}
S^{i}(\psi)+h(\psi)\leq\inf_{\phi\in\mathcal{C}([0,T];\mathbb{R}^{m})}\left[
S^{i}(\phi)+h(\phi)\right]  +\eta<\infty.
\label{Eq:NearlyOptimalTrajectoryRegime1}%
\end{equation*}

Let us consider first Regimes $1$ and $2$ and Regime 3 in dimension $1$. By Lemma \ref{T:MainTheorem21}, the local rate function of $S^{i}$ is $L_{i}(x,\beta)$ given by (\ref{Def:LocalRateFunction}).
By Theorem \ref{T:ExplicitSolutionRegime1}, it is clear that $L_{1}(x,\beta)$ is
continuous and finite at each $(x,\beta)\in\mathbb{R}^{2m}$. As in Theorem 6.3 in \cite{DupuisSpiliopoulos} the same is true for Regime $2$ and as in  Theorem 7.2 of \cite{DupuisSpiliopoulos} for Regime $3$ in dimension 1.
Thus, a
standard mollification argument, allows us to further assume that $\dot{\psi}$ is
piecewise constant (see for example Lemmas 6.5.3 and 6.5.5 in Subsection $6.5$ of
\cite{DupuisEllis}). Then, by Theorem \ref{T:ExplicitSolutionRegime1}, we obtain that with $\beta=\dot{\psi}_{t}$, there is
$\bar{u}_{\dot{\psi}_{t}}(x,y)=(\bar{u}_{1,\dot{\psi}_{t}}(x,y),\bar{u}_{2,\dot{\psi}_{t}}(x,y))$ that is bounded, continuous in $x$ and Lipschitz continuous  in
$y$, and piecewise constant in $t$ such that
\begin{equation}
\bar{u}_{\dot{\psi}_{t}}(x,\cdot)\in\text{argmin}_{v}\left\{  \frac{1}{2}\int_{\mathcal{Y}%
}\left\Vert v(y)\right\Vert ^{2}\mu(dy)\hspace{0.1cm}:(v,\mu)\in
\mathcal{A}_{x,\dot{\psi}_{t}}^{i}\right\}  .
\label{Eq:OptimalControlRegime2_1}%
\end{equation}

Let us denote by $\bar{\mu}_{\bar{u}}(dy)$ the unique invariant
measure corresponding to the operator $\mathcal{L}_{\bar{u},x}^{i}$ for this particular control $\bar{u}$ (which exists due to Theorem \ref{T:ExplicitSolutionRegime1}). The feedback type of control used to prove the upper bound  is then%
\[
\bar{u}^{\epsilon}(t)=\bar{u}_{\dot{\psi}_{t}}\left(\bar{X}^{\epsilon}(t),\bar
{Y}^{\epsilon}(t)\right)=\left(\bar{u}_{1,\dot{\psi}_{t}}\left(\bar{X}^{\epsilon}(t),\bar
{Y}^{\epsilon}(t)\right),\bar{u}_{2,\dot{\psi}_{t}}\left(\bar{X}^{\epsilon}(t),\bar
{Y}^{\epsilon}(t)\right)\right).
\]

By Condition \ref{A:Assumption1} and since $\bar{u}$ is continuous in $x$ and $y$, equation (\ref{Eq:Main2}) has a strong solution
 with $\left(u_{1}(t),u_{2}(t)\right)=\bar{u}^{\epsilon}(t)$.

Then, standard averaging theory (e.g., Section 6, Chapter 3 of \cite{BLP}) and the fact that
$\bar{\mu}_{\bar{u}_{\dot{\psi}_{t}}(x,\cdot)}(\cdot)$ is weakly continuous in $x$ (Theorem
\ref{T:ExplicitSolutionRegime1}) and piecewise continuous in $t$ we have
that $\bar{X}^{\epsilon}\overset{\mathcal{D}}{\rightarrow}\bar{X}$, where
\begin{equation*}
\bar{X}_{t}=x_{0}+\int_{0}^{t}\int_{\mathcal{Y}}\lambda_{i}\left(  \bar{X}%
_{s},y,\bar{u}_{\dot{\psi}_{s}}(\bar{X}_{s},y)\right)  \bar{\mu}_{\bar{u}_{\dot{\psi}_{s}}(\bar{X}_{s}%
,\cdot)}(dy)ds.\label{Eq:LimitingXReg2}%
\end{equation*}

Then for $\psi$ such that
$\psi_{0}=x_{0}$,   (\ref{Eq:OptimalControlRegime2_1}) and the definition of $\mathcal{A}_{x,\dot{\psi}_{t}}^{i}$ gives us that
\[
\bar{X}_{t}=x_{0}+\int_{0}^{t}\dot{\psi}_{s}ds=\psi_{t}\hspace{0.2cm}\text{
for any }t\in\lbrack0,T]\text{, w.p.}1.
\]
Therefore, by the representation formula  (\ref{Eq:VariationalRepresentation}) and (\ref{Eq:OptimalControlRegime2_1}) we finally obtain that
\begin{align}
\limsup_{\epsilon\downarrow0}\left[  -\epsilon\ln\mathbb{E}_{x_{0},y_{0}}\left[
\exp\left\{  -\frac{h(X^{\epsilon})}{\epsilon}\right\}  \right]  \right]   &
=\limsup_{\epsilon\downarrow0}\inf_{(u_{1},u_{2})}\mathbb{E}_{x_{0},y_{0}}\left[  \frac{1}%
{2}\int_{0}^{T}\left[\left\Vert u_{1}(t)\right\Vert ^{2}+\left\Vert u_{2}(t)\right\Vert ^{2}\right]dt+h(\bar{X}^{\epsilon
})\right]  \nonumber\\
&  \leq\limsup_{\epsilon\downarrow0}\mathbb{E}_{x_{0},y_{0}}\left[  \frac{1}{2}%
\int_{0}^{T}\left\Vert \bar{u}^{\epsilon}(t)\right\Vert ^{2}dt+h(\bar
{X}^{\epsilon})\right]  \nonumber\\
&  =\mathbb{E}_{x_{0},y_{0}}\left[  S^{i}(\bar{X})+h(\bar{X})\right]  \nonumber\\
&  \leq\inf_{\phi\in\mathcal{C}([0,T];\mathbb{R}^{m})}\left[  S^{i}(\phi
)+h(\phi)\right]  +\eta.\nonumber
\end{align}
Since $\eta$ is arbitrary, we are done. The proof of part (iii) for Regime 3 follows as in Section 7 of \cite{DupuisSpiliopoulos}, using the fact that there is (by Theorem \ref{T:ExplicitSolutionRegime1} below) a $\mathrm{P}\in\mathcal{A}_{\beta}^{3,r}$ that achieves the infimum in (\ref{Def:ErgodicControlConstraints1_1Relaxed}).
\end{proof}

\section{Importance Sampling}\label{S:IS}

The purpose of this section is to utilize the large deviations results of Section \ref{S:LDP} in order to obtain asymptotically efficient importance sampling schemes for
 quantities like (\ref{Eq:ToBeEstimated}).  Simulation problems involving rare events
unavoidably  have a number of mathematical and computational challenges. As it is well
known, standard Monte Carlo sampling techniques perform very poorly in that
the relative errors under a fixed computational effort grow rapidly as the
event becomes  more rare. Rare event estimation problems for  systems of fast and slow motion present
extra difficulties due to the underlying fast motion  and its
interaction with the intensity of the noise $\epsilon$.
In particular, one needs to take into account the solution to the appropriate cell problem associated with the homogenization
theory of HJB equations in order to guarantee asymptotic optimality. Related simulation results are provided in
\cite{DupuisSpiliopoulosWang,DupuisSpiliopoulosWang2} for the special case $f(x,y)=b(x,y)=-\nabla Q(y)$, $g(x,y)=c(x,y)=-\nabla V(x)$, $\sigma(x,y)=\tau_{1}(x,y)=\textrm{constant}$
and $\tau_{2}(x,y)=0$.

We start by reviewing general things about importance sampling adjusting the discussion to our setting of interest.
Consider a bounded continuous function
$h:\mathbb{R}^{m}\mapsto\mathbb{R}$ and suppose that one is interested in
estimating
\[
\theta(\epsilon)\doteq\mathrm{E}[e^{-\frac{1}{\epsilon}h(X^{\epsilon}%
(T))}|X^{\epsilon}(t_{0})=x_{0}, Y^{\epsilon}(t_{0})=y_{0}]
\]
by Monte Carlo, where the pair of slow and fast motion $(X^{\epsilon}, Y^{\epsilon})$ has initial point $X^{\epsilon}(t_{0})=x_{0}, Y^{\epsilon}(t_{0})=y_{0}$. For Regime $i=1,2,3$, let
\begin{equation}
G_{i}(t_{0},x_{0})\doteq\inf_{\phi\in\mathcal{C}([t_{0},T];\mathbb{R}^{m}),\phi(t_{0})=x_{0}}\left[
S_{t_{0}T}^{i}(\phi)+h(\phi(T))\right]. \label{Eq:VarProb}
\end{equation}
As we shall see below, under regularity conditions, the function $G_{i}(t,x)$ satisfies a PDE of HJB type. Now, depending on the regime of interaction, the contraction principle implies
\begin{equation}
\lim_{\epsilon\rightarrow0}-\epsilon\log\theta(\varepsilon)=G_{i}(t_{0},x_{0}).
\label{LDPprinciple1}%
\end{equation}
Notice that the limit is independent of the initial point $y_{0}$ of the fast motion $Y^{\epsilon}$. This is due to the averaging that takes place, as we shall also see later on in the rigorous proofs.

Let $\Gamma^{\epsilon}(t_{0},x_{0},y_{0})$ be any unbiased estimator of $\theta(\epsilon)$
that is defined on some probability space with probability measure
$\bar{\mathrm{P}}$. With $\bar{\mathrm{E}}$ denoting the expectation operator associated with
$\bar{\mathrm{P}}$ we have that $\Gamma^{\epsilon}(t_{0},x_{0},y_{0})$ is a random
variable such that
\[
\bar{\mathrm{E}}\Gamma^{\epsilon}(t_{0},x_{0},y_{0})=\theta(\epsilon).
\]

In Monte Carlo simulation, one generates a number of independent copies of
$\Gamma^{\epsilon}(t,x,y)$ and the estimate is the sample mean. The specific
number of samples required depends on the desired accuracy, which is measured
by the variance of the sample mean.  Because of
unbiasedness, minimizing the variance is equivalent to minimizing the second
moment. Jensen's inequality implies
\[
\bar{\mathrm{E}}(\Gamma^{\epsilon}(t_{0},x_{0},y_{0}))^{2}\geq(\bar{\mathrm{E}}%
\Gamma^{\epsilon}(t_{0},x_{0},y_{0}))^{2}=\theta(\epsilon)^{2}.
\]
This and  (\ref{LDPprinciple1}) say that
\[
\limsup_{\epsilon\rightarrow0}-\epsilon\log\bar{\mathrm{E}}(\Gamma^{\epsilon
}(t_{0},x_{0},y_{0}))^{2}\leq2G_{i}(t_{0},x_{0}).
\]
Hence, $2G_{i}(t_{0},x_{0})$ is the best possible rate of decay of the second moment.
If
\[
\liminf_{\epsilon\rightarrow0}-\epsilon\log\bar{\mathrm{E}}(\Gamma^{\epsilon
}(t_{0},x_{0},y_{0}))^{2}\geq2G_{i}(t_{0},x_{0}),
\]
then $\Gamma^{\epsilon}(t_{0},x_{0},y_{0})$ achieves this best decay rate, and is said to be
\textit{asymptotically optimal}.

It is important to note here that asymptotic optimality is not the only practical concern. Rare events associated with multiscale problems are rather complicated
and many times is it  very difficult to construct asymptotically optimal schemes. One way to circumvent this difficulty is by constructing appropriate sub-optimal schemes
with precise bounds on asymptotic performance.
 This is the content of Theorems \ref{T:UniformlyLogEfficientReg1}, \ref{T:UniformlyLogEfficientReg2} and \ref{T:UniformlyLogEfficientReg3}
for Regime $i=1,2,3$ respectively.

Fix the Regime $i=1,2,3$ and assume that we are given a control $\bar{u}(s,x,y;i)$ that is sufficiently smooth and bounded. Let us recall the $2\kappa-$dimensional Wiener process $Z(\cdot)=\left(W(\cdot), B(\cdot)\right)$. Consider the family of probability measures $\bar{\mathrm{P}}^{\epsilon}%
$\ defined by the change of measure
\[
\frac{d\bar{\mathrm{P}}^{\epsilon}}{d\mathrm{P}}=\exp\left\{  -\frac
{1}{2\epsilon}\int_{t_{0}}^{T}\left\Vert \bar{u}(s,X^{\epsilon}(s),Y^{\epsilon}(s);i)\right\Vert
^{2}ds+\frac{1}{\sqrt{\epsilon}}\int_{t_{0}}^{T}\left\langle \bar{u}%
(s,X^{\epsilon}(s),Y^{\epsilon}(s);i),dZ(s)\right\rangle \right\}  .
\]
By Girsanov's Theorem
\[
\bar{Z}(s)=Z(s)-\frac{1}{\sqrt{\epsilon}}\int_{t_{0}}^{s}\bar{u}(\rho,X^{\epsilon
}(\rho),Y^{\epsilon}(\rho);i)d\rho,~~~t_{0}\leq s\leq T
\]
is a Wiener process on $[t_{0},T]$ under the probability measure $\bar
{\mathrm{P}}^{\epsilon}$, and $(X^{\epsilon},Y^{\epsilon})$
satisfies $X^{\epsilon}(t_{0})=x_{0}$, $Y^{\epsilon}(t_{0})=y_{0}$ and for
$s\in(t,T]$ it is the unique strong solution of (\ref{Eq:Main2})
with $\bar{Z}(\cdot)=\left(\bar{W}(\cdot), \bar{B}(\cdot)\right)$ in
place of $Z(\cdot)=\left(W(\cdot), B(\cdot)\right)$ and
$\bar{u}(s,x,y;i)=(\bar{u}_{1}(s,x,y;i), \bar{u}_{2}(s,x,y;i))$ in
place of $u(s)=(u_{1}(s),u_{2}(s))$.

Letting
\[
\Gamma^{\epsilon}(t_{0},x_{0},y_{0})=\exp\left\{  -\frac{1}{\epsilon}h(X^{\epsilon
}(T))\right\}  \frac{d\mathrm{P}}{d\bar{\mathrm{P}}^{\epsilon}}(X^{\epsilon
}, Y^{\epsilon}),
\]
it follows easily that under $\bar{\mathrm{P}}^{\epsilon}$, $\Gamma^{\epsilon
}(t_{0},x_{0},y_{0})$ is an unbiased estimator for $\theta(\epsilon)$. The performance of
this estimator is characterized by the decay rate of its second moment
\begin{equation}
Q^{\epsilon}(t_{0},x_{0},y_{0};\bar{u})\doteq\bar{\mathrm{E}}^{\epsilon}\left[  \exp\left\{
-\frac{2}{\epsilon}h(X^{\epsilon}(T))\right\}  \left(  \frac{d\mathrm{P}%
}{d\bar{\mathrm{P}}^{\epsilon}}(X^{\epsilon}, Y^{\epsilon})\right)  ^{2}\right].
\label{Eq:2ndMoment1}%
\end{equation}

We construct asymptotically efficient importance sampling schemes by
choosing  the control $\bar{u}$ in (\ref{Eq:2ndMoment1}) such that
the behavior of the second moment $Q^{\epsilon}(t_{0},x_{0},y_{0};\bar{u})$ is
controlled. Two are the main ingredients in the construction of
$\bar{u}$:
\begin{enumerate}
\item{The gradient of a subsolution to the PDE that the function $G_{i}(t,x)$
defined in (\ref{Eq:VarProb}) satisfies. Under appropriate regularity
conditions $G_{i}(t,x)$ satisfies a PDE of Hamilton-Jacobi-Bellman (HJB) type.}
\item{The solution to the associated cell problem or in
other words the so-called corrector from the homogenization
theory of HJB equations.}
\end{enumerate}

Depending on the regime of interaction the HJB equation and
the corresponding cell problem take a different form. These will be
made precise in Subsections \ref{SS:IS_Regime1}-\ref{SS:IS_Regime3}.

As mentioned before, we work with appropriate subsolutions to the associated HJB equation.
Thus, let us now recall the notion of a subsolution to an HJB equation of the form
\begin{equation}
G_{s}(s,x)+\bar{H}(x,\nabla_{x}G(s,x))=0,\quad G(T,x)=h(x).
\label{Eq:HJBequation2}%
\end{equation}

\begin{definition}
\label{Def:ClassicalSubsolution} A function
$\bar{U}(s,x):[0,T]\times \mathbb{R}^{m}\mapsto\mathbb{R}$ is a
classical subsolution to the HJB equation (\ref{Eq:HJBequation2}) if

\begin{enumerate}
\item $\bar{U}$ is continuously differentiable,

\item $\bar{U}_{s}(s,x)+\bar{H}(x,\nabla_{x}\bar{U}(s,x))\geq0$ for every
$(s,x)\in(0,T)\times\mathbb{R}^{m}$,

\item $\bar{U}(T,x)\leq h(x)$ for $x\in\mathbb{R}^{m}$.
\end{enumerate}
\end{definition}

We will impose stronger regularity conditions on the subsolutions to
be considered than those of Definition
\ref{Def:ClassicalSubsolution}. This is convenient for the purposes
of illustrations since then the feedback control is uniformly
bounded and thus several technical problems are avoided. However, we
mention that the uniform bounds that will be assumed in Condition
\ref{Cond:ExtraReg} can be replaced by milder conditions with the
expense of working harder to establish the results.

\begin{condition}
\label{Cond:ExtraReg} $\bar{U}$ has continuous derivatives up to
order $1$ in $t$ and order $2$ in $x$, and the first and second
derivatives in $x$ are uniformly bounded.
\end{condition}
Roughly speaking, our main result is as follows.

\begin{theorem}
\label{T:UniformlyLogEfficient} Consider a bounded and continuous
function $h:\mathbb{R}^{m}\mapsto\mathbb{R}$ and assume Conditions
\ref{A:Assumption1} and under Regime $1$ assume Condition \ref{A:Assumption2}. Let $\{\left(X^{\epsilon}(s), Y^{\epsilon}(s) \right),\epsilon>0\}$ be
the solution to (\ref{Eq:Main}) for $s\in[t_{0},T]$ with initial point $(x_{0},y_{0})$ at time $t_{0}$.  Under Regime $i=1,2,3$ let $\bar{u}(s,x,y;i)$  be an appropriately defined and smooth control in terms of a subsolution $\bar{U}_{i}(s,x)$ to the HJB satisfied by $G_{i}(s,x)$ and the corrector from the corresponding cell problem. Then
\begin{equation}
\liminf_{\epsilon\rightarrow0}-\epsilon\ln Q^{\epsilon}(t_{0},x_{0},y_{0};\bar{u}(\cdot;i))\geq
G_{i}(t_{0},x_{0})+\bar{U}_{i}(t_{0},x_{0}). \label{Eq:GoalSubsolution}%
\end{equation}
\end{theorem}
Once we have established a Theorem like \ref{T:UniformlyLogEfficient}, we can make a claim for estimating probabilities of the form $\mathrm{P}_{t,x_{0},y_{0}}[X^{\epsilon}(T)\in A]$ as well. The claim of the following proposition  is not readily covered by Theorem \ref{T:UniformlyLogEfficient}, since the function $h$ is
neither bounded nor continuous. However, by an approximating argument analogous to \cite{DupuisWang2} the claim can be established. We omit the details of the proof and only present the statement.

\begin{proposition}
\label{P:UniformlyLogEfficient1} Assume Conditions
\ref{A:Assumption1} and \ref{Cond:ExtraReg} and under Regime $1$ assume Condition \ref{A:Assumption2}. Let $\{\left(X^{\epsilon}, Y^{\epsilon} \right),\epsilon>0\}$ be
the solution to (\ref{Eq:Main}) with initial point $(t_{0},x_{0},y_{0})$. Under Regime $i$,  let $A\subset\mathbb{R}^{m}$ be a regular set with respect to the action functional $S^{i}$ and the initial point $(t_{0},x_{0},y_{0})$, i.e., the infimum of $S^{i}$ over the closure $\bar{A}$ is the same as
the infimum over the interior $A^{o}$. 
Let
\[
h(x)=%
\begin{cases}
0 & \text{if }x\in A\\
+\infty & \text{if }x\notin A.
\end{cases}
\]
Let $\bar{u}(s,x,y;i)$  be an appropriately defined and smooth control as in Theorem \ref{T:UniformlyLogEfficient}. Then (\ref{Eq:GoalSubsolution}) holds.
\end{proposition}

Notice that the lower asymptotic bound of Theorem \ref{T:UniformlyLogEfficient} and Proposition \ref{P:UniformlyLogEfficient1} is independent of the initial point $y_{0}$ of the fast component $Y^{\epsilon}$. This is due to averaging.

\begin{remark}
Since $\bar{U}_{i}$ is a subsolution, we get that $\bar{U}_{i}(s,x)\leq
G_{i}(s,x)$ everywhere. By (\ref{Eq:GoalSubsolution}) this implies that  the scheme is asymptotically optimal if $\bar
{U}_{i}(t_{0},x_{0})=G_{i}(t_{0},x_{0})$ at the starting point $(t_{0},x_{0})$. Standard Monte Carlo corresponds to choosing the subsolution $\bar{U}_{i}=0$. Hence, any subsolution with value at the origin $(t_{0},x_{0})$ such that
$$0\ll\bar{U}_{i}(t_{0},x_{0})\leq G_{i}(t_{0},x_{0}) $$ will have better asymptotic performance than that of standard Monte Carlo.
\end{remark}

In the next subsections we present how one can choose the controls $\bar{u}(s,x,y;i)$ in terms of a subsolution $\bar{U}$ and its corresponding cell problem such that the bound mentioned in Theorem \ref{T:UniformlyLogEfficient} is attained. The situation is subtle here due to the multiscale aspect of the problem.

\subsection{Importance sampling for Regime 1.}\label{SS:IS_Regime1}

In this subsection we construct asymptotically efficient importance sampling schemes for Regime $1$. In Regime 1, the form of the Hamiltonian $\bar{H}(x,p)$ in
(\ref{Eq:HJBequation2}) is naturally suggested by the calculus of
variation problem
(\ref{Eq:VarProb}) and the explicit formula of the rate function $S^{1}_{tT}%
(\phi)$ in Theorem \ref{T:MainTheorem3}:
\begin{equation}
\bar{H}(x,p)=\left\langle r(x),p\right\rangle -\frac{1}{2}\langle p,q(x)p\rangle.
\label{Eq:ControlFormHJB}%
\end{equation}
In fact, under mild conditions $G_{1}$ from (\ref{Eq:VarProb}) is the unique viscosity
solution to (\ref{Eq:HJBequation2}) with $\bar{H}(x,p)$ defined by
(\ref{Eq:ControlFormHJB}).

We have the following Theorem.

\begin{theorem}
\label{T:UniformlyLogEfficientReg1} Let $\{\left(X^{\epsilon}(s), Y^{\epsilon}(s) \right),\epsilon>0\}$ be
the solution to (\ref{Eq:Main}) for $s\in[t_{0},T]$ with initial point $(x_{0},y_{0})$ at time $t_{0}$. Consider a bounded and continuous
function $h:\mathbb{R}^{m}\mapsto\mathbb{R}$ and assume Conditions
\ref{A:Assumption1}, \ref{A:Assumption2} and \ref{Cond:ExtraReg}. Let $\bar{U}_{1}(s,x)$ be a subsolution to the associated HJB equation. Define the feedback control $\bar{u}(s,x,y;1)=\left(\bar{u}_{1}(s,x,y;1), \bar{u}_{2}(s,x,y;1)\right)$ by

 \begin{equation*}
\bar{u}(s,x,y;1)=\left(-\left(\sigma+\frac{\partial\chi}{\partial y}\tau_{1}\right)^{T}(x,y)\nabla_{x}\bar{U}_{1}(s,x), -\left(\frac{\partial\chi}{\partial y}\tau_{2}\right)^{T}(x,y)\nabla_{x}\bar{U}_{1}(s,x)\right)\label{Eq:feedback_controlReg1}
\end{equation*}

 Then the conclusion of Theorem \ref{T:UniformlyLogEfficient} holds, i.e.
\begin{equation*}
\liminf_{\epsilon\rightarrow0}-\epsilon\ln Q^{\epsilon}(t_{0},x_{0},y_{0};\bar{u}(\cdot;1))\geq
G_{1}(t_{0},x_{0})+\bar{U}_{1}(t_{0},x_{0}). \label{Eq:GoalRegime1Subsolution}%
\end{equation*}
\end{theorem}

Before proceeding with the proof, we notice that the feedback control (\ref{Eq:feedback_controlReg1}) is essentially implied by the solution to the variational problem associated with the local rate function in the definition of the action functional for Regime 1, Theorem \ref{T:ExplicitSolutionRegime1}.
\begin{proof}

Note that under the given conditions $\bar{u}(s,x,y;1)$ is Lipschitz continuous in
$(x,y)$, continuous in $(t,x,y)$, and uniformly bounded. For notational convenience, we omit the subscript $1$ from $G_{1}$ and $\bar{u}_{1}$ and we write $(t,x,y)$ in place of $(t_{0},x_{0},y_{0})$.

Boundedness of $h$ and $\bar{u}$ imply by  the representation formula (\ref{Eq:VariationalRepresentation})  and by the Lemma 4.3 of \cite{DupuisSpiliopoulosWang} that
\begin{align}
\lefteqn{-\epsilon\log Q^{\epsilon}(t,x,y;\bar{u})}\label{Eq:ToBeBounded}\\
&  =\inf_{v\in\mathcal{A}}\mathrm{E}\left[  \frac{1}{2}\int_{t}^{T}\left\Vert
v(s)\right\Vert ^{2}ds-\int_{t}^{T}\Vert\bar{u}(s,\hat{X}^{\epsilon}(s),\hat{Y}^{\epsilon}
(s))\Vert^{2}ds+2h(\hat{X}^{\epsilon}(T))\right]  ,\nonumber
\end{align}
where $v(s)=(v_{1}(s),v_{2}(s))$,
$\bar{u}(s,x,y;1)=(\bar{u}_{1}(s,x,y;1),\bar{u}_{2}(s,x,y;1))$ and
$(\hat{X},\hat{Y})$ satisfying

\begin{eqnarray}
d\hat{X}^{\epsilon}(s)&=&\left[  \frac{\epsilon}{\delta}b\left(  \hat{X}^{\epsilon}(s)%
,\hat{Y}^{\epsilon}(s)\right)+\bar{c}\left(  \hat{X}^{\epsilon}(s)%
,\hat{Y}^{\epsilon}(s)\right)+\sigma\left(  \hat{X}_{t}^{\epsilon},\hat{Y}_{t}^{\epsilon}\right)  v_{1}(s)\right]   ds+\sqrt{\epsilon}%
\sigma\left(  \hat{X}^{\epsilon}(s),\hat{Y}^{\epsilon}(s)\right)
dW(s), \nonumber\\
d\hat{Y}^{\epsilon}(s)&=&\frac{1}{\delta}\left[  \frac{\epsilon}{\delta}f\left(  \hat{X}^{\epsilon}(s)
,\hat{Y}^{\epsilon}(s)\right)  +\bar{g}\left(  \hat{X}^{\epsilon}(s)
,\hat{Y}^{\epsilon}(s)\right)+\tau_{1}\left(  \hat{X}^{\epsilon}(s)
,\hat{Y}^{\epsilon}(s)\right)v_{1}(s)+\tau_{2}\left(  \hat{X}^{\epsilon}(s)
,\hat{Y}^{\epsilon}(s)\right)v_{2}(s)\right]   ds\nonumber\\
& &\hspace{3.5cm}+\frac{\sqrt{\epsilon}}{\delta}\left[
\tau_{1}\left(  \hat{X}^{\epsilon}(s),\hat{Y}^{\epsilon}(s)\right)
dW(s)+\tau_{2}\left(  \hat{X}^{\epsilon}(s),\hat{Y}^{\epsilon}(s)\right)dB(s)\right],\label{Eq:Main3}\\
\hat{X}^{\epsilon}(0)&=&x_{0},\hspace{0.2cm}\hat{Y}^{\epsilon}(0)=y_{0}\nonumber
\end{eqnarray}

with
\begin{eqnarray}
\bar{c}\left(  s,x,y\right)&=&c(x,y)-\sigma(x,y)\bar{u}_{1}(s,x,y;1)\nonumber\\
\bar{g}\left(  s,x,y\right)&=&g(x,y)-\tau_{1}(x,y)\bar{u}_{1}(s,x,y;1)-\tau_{2}(x,y)\bar{u}_{2}(s,x,y;1)
\end{eqnarray}

The next step is to take the limit infimum in the representation
(\ref{Eq:ToBeBounded}). The right hand side of
(\ref{Eq:ToBeBounded}) can be bounded by below in the limit
$\epsilon\downarrow 0$ using statement (ii) of Theorem
\ref{T:MainTheorem2} with two differences. The first difference is
that the functions $c,g$ in the definition of the first component of
appropriate viable pair $(\lambda_{1},\mathcal{L}^{1})$, see
Definition \ref{Def:ThreePossibleFunctions}, are replaced with
$\bar{c},\bar{g}$. Using Theorem \ref{T:ExplicitSolutionRegime1},
the local rate function takes the form
\begin{equation*}
L_{1}(x,\beta)=\frac{1}{2}\left(\beta-\bar{r}(s,x)\right)^{T}q^{-1}(x)\left(\beta-\bar{r}(s,x)\right)
\end{equation*}
where $\bar{r}(s,x)=r(x)-\int_{\mathcal{Y}}\left(\sigma
\bar{u}_{1}+\frac{\partial \chi}{\partial
y}\tau_{1}\bar{u}_{1}+\frac{\partial \chi}{\partial
y}\tau_{2}\bar{u}_{2} \right)(s,x,y)\mu(dy|x)$. Here $\mu(dy|x)$ is the invariant measure defined in Condition \ref{A:Assumption2}. This takes care of the limit of first term on the right hand side of (\ref{Eq:ToBeBounded}). The second
difference is the presence of the additional integral term $-\int_{t}^{T}\Vert\bar{u}(s,\hat{X}^{\epsilon}(s),\hat{Y}^{\epsilon}(s))\Vert^{2}ds$.
 Using classical averaging arguments, see
\cite{BLP}, appropriately modified to treat controlled processes, as in Lemma 3.2
in \cite{DupuisSpiliopoulos}, this term can be replaced in the limit as $\epsilon\downarrow 0$
by its averaged version with respect to $\mu(dy|x)$. This takes care of the limit of second term on the right hand side of (\ref{Eq:ToBeBounded}).
Putting these together, we have
\begin{align}
\liminf_{\epsilon\rightarrow0}-\epsilon\log Q^{\epsilon}(t,x;\bar
{u})& \geq \inf_{\phi,\phi(t)=x}\left[ \int_{t}^{T}L_{1}(\phi(s),\dot{\phi}(s))ds\right.\nonumber\\
& \left.-\int_{t}^{T}\int_{\mathcal{Y}}\left[\left\Vert \bar
{u}_{1}(s,\phi(s),y)\right\Vert ^{2}+\left\Vert \bar
{u}_{2}(s,\phi(s),y)\right\Vert ^{2}\right]\mu(dy|\phi(s))ds+2h(\phi(T))\right]
\end{align}
By recalling the formula for $\bar{u}=(\bar{u}_{1},\bar{u}_{2})$ we have for $\phi\in\mathcal{AC}([t,T];\mathbb{R}^{m})$
\begin{equation*}
\int_{t}^{T}\int_{\mathcal{Y}}\left[\left\Vert \bar
{u}_{1}(s,\phi(s),y)\right\Vert ^{2}+\left\Vert \bar
{u}_{2}(s,\phi(s),y)\right\Vert ^{2}\right]\mu(dy|\phi(s))ds= \int_{t}^{T}\langle\nabla_{x}\bar{U}(s,\phi(s)),q(x)\nabla
_{x}\bar{U}(s,\phi(s))\rangle ds
\end{equation*}

Thus, we have

\begin{align}
\lefteqn{\liminf_{\epsilon\rightarrow0}-\epsilon\log Q^{\epsilon}(t,x;\bar
{u})}\nonumber\\
& \geq \inf_{\phi,\phi(t)=x}\left[ \frac{1}{2}\int_{t}^{T}\left\Vert
\dot{\phi}(s)- r(\phi(s))-\int_{\mathcal{Y}}\left(\sigma
\bar{u}_{1}+\frac{\partial \chi}{\partial
y}\tau_{1}\bar{u}_{1}+\frac{\partial \chi}{\partial
y}\tau_{2}\bar{u}_{2} \right)(s,\phi(s),y)\mu(dy|\phi(s))\right\Vert
_{q^{-1}(\phi(s))}^{2}ds\right. \nonumber\\
& \hspace{4cm}\left.  -\int_{t}^{T}\int_{\mathcal{Y}}\left[\left\Vert \bar
{u}_{1}(s,\phi(s),y)\right\Vert ^{2}+\left\Vert \bar
{u}_{2}(s,\phi(s),y)\right\Vert ^{2}\right]\mu(dy|\phi(s))ds+2h(\phi(T))\right]\nonumber\\
&= \inf_{\phi\in\mathcal{AC}([t,T];\mathbb{R}^{m}),\phi(t)=x}\left[\int_{t}^{T}\left[\frac{1}{2}\left\Vert \dot{\phi}(s)-r(\phi(s))\right\Vert
_{q^{-1}(\phi(s))}^{2}-\langle\dot{\phi}(s)-r(\phi
(s)),\nabla_{x}\bar{U}(s,\phi(s))\rangle \right]ds\right.\nonumber\\
&  \hspace{4cm}\left.~~~~~-\frac{1}{2}\int_{t}^{T}\langle\nabla_{x}\bar{U}(s,\phi(s)),q(x)\nabla
_{x}\bar{U}(s,\phi(s))\rangle ds+2h(\phi(T))\right]=\nonumber\\
&= \inf_{\phi\in\mathcal{AC}([t,T];\mathbb{R}^{m}),\phi(t)=x}\left[S_{tT}^{1}(\phi)+2h(\phi(T))\right.\nonumber\\
&  \hspace{3cm}\left.-\int_{t}^{T}\left(\langle\dot{\phi}(s)-r(\phi
(s)),\nabla_{x}\bar{U}(s,\phi(s))\rangle+\frac{1}{2}\langle\nabla_{x}\bar{U}(s,\phi(s)),q(x)\nabla
_{x}\bar{U}(s,\phi(s))\rangle\right) ds\right]
\label{Eq:VarLim2}
\end{align}
In the first equality  we have used the definition of
$\bar{u}=(\bar{u}_{1},\bar{u}_{2})$ whereas in the second equality we used the definition of the action functional by Theorem
\ref{T:MainTheorem3}.

Given an arbitrary $\phi\in\mathcal{AC}([t,T];\mathbb{R}^{m})$ with
$\phi(t)=x$, the subsolution property implies that

\begin{align*}
& -\langle\dot{\phi}(s)-r(\phi(s)),\nabla_{x}\bar{U}(s,\phi(s))\rangle
-\frac{1}{2}\langle\nabla_{x}\bar{U}(s,\phi(s)),q(\phi(s))\nabla_{x}\bar
{U}(s,\phi(s))\rangle\nonumber\\
&\hspace{8cm}\geq -\partial_{t}\bar{U}(s,\phi(s))-\langle\nabla_{x}\bar{U}(s,\phi(s)),\dot{\phi}(s)\rangle\nonumber\\
&\hspace{8cm}=-\frac{d}{ds}\bar{U}(s,\phi(s))  \nonumber
\end{align*}
Let us now integrate both sides on $[t,T]$. Using  the terminal condition $\bar{U}(T,x)\leq h(x)$, we have
\begin{equation*}
 -\int_{t}^{T}\left(\langle\dot{\phi}(s)-r(\phi(s)),\nabla_{x}\bar{U}(s,\phi(s))\rangle
+\frac{1}{2}\langle\nabla_{x}\bar{U}(s,\phi(s)),q(\phi(s))\nabla_{x}\bar
{U}(s,\phi(s))\rangle\right)ds \geq -h(\phi(T))+\bar{U}(t,x)
\end{equation*}

Thus, the right hand side of (\ref{Eq:VarLim2}) is bounded from below
by
\[
\inf_{\phi\in\mathcal{AC}([t,T];\mathbb{R}^{m}),\phi(t)=x}\left[S_{tT}^{1}(\phi)+h(\phi(T))\right]+\bar{U}(t,x).
\]
Thus, since by definition  $G(t,x)=\inf_{\phi\in\mathcal{AC}([t,T];\mathbb{R}^{m}),\phi(t)=x}\left[S_{tT}^{1}(\phi)+h(\phi(T))\right]$ we can conclude that
\[
\liminf_{\epsilon\rightarrow0}-\epsilon\log Q^{\epsilon}(t,x;\bar{u})\geq
G(t,x)+\bar{U}(t,x).
\]
This concludes the proof. \hfill
\end{proof}

\subsection{Importance sampling for Regime 2.}\label{SS:IS_Regime2}
Let us now study the construction of efficient importance samplings
for Regime 2. The situation here is more subtle than it is for Regime $1$. This is also seen
from the large deviations principle, Theorem \ref{T:MainTheorem21}.
The key difference between the LDP for Regimes 1 and 2 is that
$\mathcal{L}_{z_{1},z_{2},x}^{2}$ depends on $(z_{1},z_{2})$, while
$\mathcal{L}_{x}^{1}$ did not. This means that relations between the
elements of a viable pair are more complex, and in particular that
the joint distribution of the control $(z_{1},z_{2})$ and fast
variable $y$ is important. Thus, in contrast to Regime 1 where the
action functional can be written down explicitly, in the case of
Regime 2 the formula of the action functional is in terms of value
function to a variational problem.

This implicit characterization partially carries over to the
importance sampling. The optimal control is again in terms of a
corresponding cell problem as it was for Regime 1 (recall the cell
problem (\ref{Eq:CellProblem}) for Regime 1). The difference here is
that the cell problem is defined implicitly rather than explicitly.
As we discuss in Section \ref{S:HJB}, this is related to
the homogenization theory of HJB equations.

In what follows, the subscript $\gamma$ is to emphasize the dependence on $\gamma$ (see (\ref{Def:ThreePossibleRegimes})). Define
\begin{eqnarray}
H_{\gamma}(x,y,p,q,P,Q,R)&=&\inf_{u_{1},u_{2}\in\mathcal{Z}}\left[ \frac{1}{2}\sigma\sigma^{T}:P+\gamma\frac{1}{2}\left(\tau_{1}\tau_{1}^{T}+\tau_{2}\tau_{2}^{T}\right):Q+ \gamma \tau_{1}\sigma^{T}:R+\left<\gamma b+c+\sigma u_{1},p\right> \right.\nonumber\\
& &\left. +\left<\gamma f+g+\tau_{1}u_{1}+\tau_{2}u_{2},q\right>+\frac{1}{2}\left\Vert u_{1}\right\Vert^{2}+\frac{1}{2}\left\Vert u_{2}\right\Vert^{2}\right]\nonumber\\
& =&\frac{1}{2}\sigma\sigma^{T}:P+\gamma\frac{1}{2}\left(\tau_{1}\tau_{1}^{T}+\tau_{2}\tau_{2}^{T}\right):Q+ \gamma \tau_{1}\sigma^{T}:R+\left<\gamma b+c,p\right> \nonumber\\
& &+ \left<\gamma f+g,q\right>-\frac{1}{2}\left\Vert \sigma^{T}p+\tau_{1}^{T}q\right\Vert^{2}-\frac{1}{2}\left\Vert \tau^{T}_{2}q\right\Vert^{2}
\label{Eq:CellProblemRegime2_0}
\end{eqnarray}
The infimum in (\ref{Eq:CellProblemRegime2_0}) is attained for
\begin{equation}
u_{1}=-\sigma^{T}(x,y) p-\tau_{1}^{T}(x,y)q \textrm{ and } u_{2}=-\tau^{T}_{2}(x,y)q
\end{equation}
The control $u=(u_{1},u_{2})$ motivates the asymptotically optimal change of measure in Theorem \ref{T:UniformlyLogEfficientReg2}. Let us now define the associated HJB equation of interest together with the associated cell problem. We start with the cell problem. For each fixed $(x,p)$ consider the unique value $\bar{H}_{\gamma}(x,p)$ such that there is a periodic solution $\xi$ to the cell problem
\begin{equation}
H_{\gamma}(x,y,p,\nabla_{y}\xi_{\gamma},0,\nabla^{2}_{y}\xi_{\gamma},0)=\bar{H}_{\gamma}(x,p)\label{Eq:CellProblemRegime2_1}
\end{equation}
The unknown in (\ref{Eq:CellProblemRegime2_1}) is the pair $\left(\xi_{\gamma},\bar{H}_{\gamma}\right)$. As it can be obtained by Theorem II.2 in  \cite{ArisawaLions}, $\xi_{\gamma}$ is the unique (up to an additive constant) periodic solution to (\ref{Eq:CellProblemRegime2_1})
such that $\xi_{\gamma}\in\mathcal{C}^{2}(\mathbb{R}^{d-m})$. Moreover, $\bar{H}_{\gamma}(x,p)$ is continuous in $x$ and concave in $p$ (see Propositions $11$ and $12$ in \cite{AlvarezBardi2001}).

Consider then the HJB equation (\ref{Eq:HJBequation2}) with $\bar{H}(x,p)$ replaced by $\bar{H}_{\gamma}(x,p)$. Under the standing assumptions,
this HJB equation has a unique viscosity solution which we denote by  $G_{2}(s,x)$. Actually, under mild conditions the value function of the variational problem (\ref{Eq:VarProb}) is this unique viscosity solution. This can be derived as in \cite{AlvarezBardi2001} and it will be recalled in Section \ref{S:HJB}.

In accordance to what we did for Regime 1, we consider a classical
subsolution to that HJB equation, which we denote by
$\bar{U}_{2}(s,x)$, where the Hamiltonian is $\bar{H}_{\gamma}(x,p)$.
 Notice that $\xi_{\gamma}$  depends on the triple
$(x,y,p)$ with $(x,p)$ seen as parameters. In the computations $p$
will be substituted by the gradient of the subsolution
$\nabla_{x}\bar{U}_{2}(s,x)$. So, in principle $\xi_{\gamma}$
and $\bar{U}_{2}(s,x)$ are coupled. This coupling is in line
with the coupling that appears in large deviations, see Theorem
\ref{T:MainTheorem21}.

Similarly to what we did for Regime $1$, we impose stronger
regularity conditions. This is done to ease exposition. In
particular, the following condition guarantees the feedback control
used in importance sampling is uniformly bounded and that we can
apply It\^{o} formula directly without approximations.  Thus, a
number of technicalities are circumvented.

\begin{condition}
\label{Cond:ExtraRegReg2} $\bar{U}_{2}$ has continuous derivatives up to order $1$ in
$t$ and order $2$ in $x$, and the first and second derivatives in $x$ are
uniformly bounded. Similarly,
$\xi_{\gamma}$ is twice continuous differentiable in $(x,y,p)$, periodic with respect to $y$ and all of the mixed derivatives up to order $2$ are bounded.
\end{condition}

The following verification theorem is the analogous of Theorem \ref{T:UniformlyLogEfficientReg1} for Regime 2.

\begin{theorem}
\label{T:UniformlyLogEfficientReg2} Let $\{\left(X^{\epsilon}(s),
Y^{\epsilon}(s) \right),\epsilon>0\}$ be the solution to
(\ref{Eq:Main}) for $s\in[t_{0},T]$ with initial point $(x_{0},y_{0})$ at time $t_{0}$. Assume that we are considering Regime $2$. Consider a bounded and continuous function
$h:\mathbb{R}^{m}\mapsto\mathbb{R}$ and assume Conditions
\ref{A:Assumption1}.
Let $\xi_{\gamma}(x,y,p)$ be the unique (up to a constant) periodic solution to the cell problem (\ref{Eq:CellProblemRegime2_1}) and $\bar{U}_{2}(s,x)$  be a classical subsolution
according to Definition \ref{Def:ClassicalSubsolution} and assume
Condition \ref{Cond:ExtraRegReg2}.  Define the control $\bar{u}(s,x,y;2)=(\bar{u}_{1}(s,x,y;2),\bar{u}_{2}(s,x,y;2))$ by
\begin{equation*}
\bar{u}(s,x,y;2)=\left(-\sigma^{T}(x,y)\nabla_{x}\bar{U}_{2}(s,x)-\tau_{1}^{T}(x,y)\nabla_{y}\xi_{\gamma}(x,y,\nabla_{x}\bar{U}_{2}\left(s,x\right)),
-\tau_{2}^{T}(x,y)\nabla_{y}\xi_{\gamma}\left(x,y,\nabla_{x}\bar{U}_{2}(s,x)\right)\right)\label{Eq:feedback_controlReg2}
\end{equation*}
Then the conclusion of Theorem \ref{T:UniformlyLogEfficient} holds, i.e.
\begin{equation*}
\liminf_{\epsilon\rightarrow0}-\epsilon\ln Q^{\epsilon}(t_{0},x_{0},y_{0};\bar{u}(\cdot;2))\geq
G_{2}(t_{0},x_{0})+\bar{U}_{2}(t_{0},x_{0}). \label{Eq:GoalRegime2Subsolution}%
\end{equation*}
\end{theorem}

\begin{proof}
For notational convenience, we omit the subscripts $2$ and $\gamma$ from
$G_{2},\bar{U}_{2},\xi_{\gamma}$ and $\bar{H}_{\gamma}$. Also we write
$\xi(s,x,y)$ in place of
$\xi\left(x,y,\nabla_{x}\bar{U}(s,x)\right)$ and $(t,x,y)$ in place of $(t_{0},x_{0},y_{0})$ for the initial point.

The first step is to write, as in Regime $1$, that
\begin{align}
\lefteqn{-\epsilon\log Q^{\epsilon}(t,x,y;\bar{u}(;2))}\label{Eq:GoalRegime2_3}\\
&  =\inf_{v\in\mathcal{A}}\mathrm{E}\left[  \frac{1}{2}\int_{t}^{T}\left\Vert
v(s)\right\Vert ^{2}ds-\int_{t}^{T}\Vert\bar{u}(s,\hat{X}^{\epsilon}(s),\hat{Y}^{\epsilon}
(s);2)\Vert^{2}ds+2h(\hat{X}^{\epsilon}(T))\right]  ,\nonumber
\end{align}
where $(\hat{X},\hat{Y})$  satisfies (\ref{Eq:Main3}) with
$v(s)=(v_{1}(s),v_{2}(s))$ and
$\bar{u}(s,x,y;2)=(\bar{u}_{1}(s,x,y;2),\bar{u}_{2}(s,x,y;2))$.

The next step is to rewrite the right hand side of
(\ref{Eq:GoalRegime2_3}). Recall the definition of the operator
$\mathcal{L}^{2}_{z,x}$ from Definition
\ref{Def:ThreePossibleOperators} with $z=(z_{1},z_{2})$. Denote by
$\mathcal{L}^{\epsilon/\delta,2}_{z,x}$ the operator
$\mathcal{L}^{2}_{z,x}$ with $\frac{\epsilon}{\delta}$ in place of
$\gamma$. We will write $\mathcal{L}^{\epsilon/\delta,2}_{0,x}$ to
denote the operator with the control variable $z=0$.

Apply It\^{o} formula to $\xi(s,x,y)$. After some term
rearrangement,  we get
\begin{equation}
-\int_{t}^{T}\mathcal{L}^{\epsilon/\delta,2}_{0,\hat{X}^{\epsilon}(s)}\xi\left(s,\hat{X}^{\epsilon}(s),\bar{Y}^{\epsilon}(s)\right)
ds=
\int_{t}^{T}\left<\nabla_{y}\xi,\tau_{1}\left(v_{1}-\bar{u}_{1}\right)+\tau_{2}\left(v_{2}-\bar{u}_{2}\right)\right>\left(s,\hat{X}^{\epsilon}(s),\hat{Y}^{\epsilon}(s)\right)ds+
R_{1}(\epsilon,v)\label{Eq:UniformlyLogEfficient2Regime2_1}
\end{equation}
where the random variable $R_{1}(\epsilon,v)$ is
\begin{eqnarray}
R_{1}(\epsilon,v)&=& \delta\left[\xi(x,x/\delta,t)-\xi\left(T,\hat{X}^{\epsilon}(T),\hat{Y}^{\epsilon}(T)\right)+\int_{t}^{T}\partial_{t}\xi\left(s,\hat{X}^{\epsilon}(s),\hat{Y}^{\epsilon}(s)\right)ds\right]+\nonumber\\
& &+
\epsilon\int_{t}^{T}\tau_{1}^{T}\sigma:\nabla_{x}\nabla_{y}\xi\left(s,\hat{X}^{\epsilon}(s),\hat{Y}^{\epsilon}(s)\right)ds\nonumber\\
&+&\delta\int_{t}^{T}\left[\left<\frac{\epsilon}{\delta}b+c+\sigma(v_{1}(s)-\bar{u}_{1}),\nabla_{x}\xi\right>\left(s,\hat{X}^{\epsilon}(s),\hat{Y}^{\epsilon}(s)\right) +\frac{\epsilon}{2}\sigma\sigma^{T}:\nabla_{xx}\xi\left(s,\hat{X}^{\epsilon}(s),\hat{Y}^{\epsilon}(s)\right)\right]ds\nonumber\\
&+&\sqrt{\epsilon}\delta\int_{t}^{T}\left<\nabla_{x}\xi,\sigma
dW(s)\right>\left(s,\hat{X}^{\epsilon}(s),\hat{Y}^{\epsilon}(s)\right)
+\sqrt{\epsilon}\int_{t}^{T}\left<\nabla_{y}\xi,\tau_{1} dW(s)+\tau_{2}dB(s)\right>\left(s,\hat{X}^{\epsilon}(s),\hat{Y}^{\epsilon}(s)\right)\nonumber
\end{eqnarray}

Under our assumptions, the random variable $R_{1}(\epsilon,v)$ converges in $L^{2}$ to zero as $\epsilon,\delta\downarrow 0$ uniformly in $v\in\mathcal{A}$.

Next, we apply It\^{o} formula to $\bar{U}(t,x)$. Omitting some
function arguments for notational convenience and using the
subsolution property for $\bar{U}$, we get
\begin{eqnarray}
h(\hat{X}^{\epsilon}(T))&\geq&\bar{U}(t,x)+
\int_{t}^{T}\left[-\bar{H}(\hat{X}^{\epsilon}(s),\nabla_{x}\bar{U})+\left<\nabla_{x}\bar{U},\frac{\epsilon}{\delta}b+c+\sigma(v_{1}(s)-\bar{u}_{1})\right>\right]
\left(s,\hat{X}^{\epsilon}(s),\hat{Y}^{\epsilon}(s)\right)ds+\nonumber\\
&+&\frac{\epsilon}{2}\int_{t}^{T}\sigma\sigma^{T}\left(\hat{X}^{\epsilon}(s),\hat{Y}^{\epsilon}(s)\right):\nabla_{x}\nabla_{x} \bar{U} \left(s,\hat{X}^{\epsilon}(s)\right)ds+\nonumber\\
& &+
\sqrt{\epsilon}\int_{t}^{T}\left<\nabla_{x}\bar{U}\left(s,\hat{X}^{\epsilon}(s)\right),\sigma\left(\hat{X}^{\epsilon}(s),\hat{Y}^{\epsilon}(s)\right)dW(s)\right>\label{Eq:UniformlyLogEfficient2Regime2_2}
\end{eqnarray}
Recalling the definition of $\bar{H}$ by (\ref{Eq:CellProblemRegime2_1}) and adding and subtracting
the term
$\int_{t}^{T}\mathcal{L}^{\epsilon/\delta,2}_{0,\hat{X}^{\epsilon}(s)}\xi\left(s,\hat{X}^{\epsilon}(s),\hat{Y}^{\epsilon}(s)\right)
ds$, relation (\ref{Eq:UniformlyLogEfficient2Regime2_2}) becomes, after using (\ref{Eq:UniformlyLogEfficient2Regime2_1}),
\begin{eqnarray}
h(\hat{X}^{\epsilon}(T))-\bar{U}(t,x)&\geq&\left(\epsilon/\delta-\gamma\right)\int_{t}^{T}\left<\nabla_{x} \bar{U}\left(s,\hat{X}^{\epsilon}(s)\right),b\left(\hat{X}^{\epsilon}(s),\hat{Y}^{\epsilon}(s)\right)\right>ds\nonumber\\
& &
+\int_{t}^{T}\left<\nabla_{x}\bar{U}\left(s,\hat{X}^{\epsilon}(s)\right), \sigma (v_{1}(s)-\bar{u}_{1})\left(s,\hat{X}^{\epsilon}(s),\hat{Y}^{\epsilon}(s)\right)\right>ds+\nonumber\\
& &+\int_{t}^{T}\left<\nabla_{y}\xi,\tau_{1}(v_{1}(s)-\bar{u}_{1})+\tau_{2}(v_{2}(s)-\bar{u}_{2})\right>ds+\nonumber\\
& &+\frac{1}{2}\int_{t}^{T}\left\Vert \bar{u}_{1}\left(s,\hat{X}^{\epsilon}(s),\hat{Y}^{\epsilon}(s);2\right)\right\Vert^{2}ds+\frac{1}{2}\int_{t}^{T}\left\Vert \bar{u}_{2}\left(s,\hat{X}^{\epsilon}(s),\hat{Y}^{\epsilon}(s);2\right)\right\Vert^{2}ds\nonumber\\
&
&+R_{1}(\epsilon,v)+R_{2}(\epsilon,v)\label{Eq:UniformlyLogEfficient2Regime2_2_1}
\end{eqnarray}
where $R_{1}(\epsilon,v)$ was defined before and $R_{2}(\epsilon,v)$
is as follows
\begin{eqnarray}
R_{2}(\epsilon,v)&=&\frac{\epsilon}{2}\int_{t}^{T}\sigma\sigma^{T}:\nabla_{x}\nabla_{x} \bar{U} \left(\hat{X}^{\epsilon}(s),\hat{Y}^{\epsilon}(s)\right)ds+
\sqrt{\epsilon}\int_{t}^{T}\left<\nabla_{x}\bar{U}\left(s,\hat{X}^{\epsilon}(s)\right),\sigma\left(\hat{X}^{\epsilon}(s),\hat{Y}^{\epsilon}(s)\right)dW(s)\right>\nonumber\\
&
&+\int_{t}^{T}\left[\mathcal{L}^{\epsilon/\delta,2}_{0,\hat{X}^{\epsilon}(s)}\xi\left(s,\hat{X}^{\epsilon}(s),\hat{Y}^{\epsilon}(s)\right)-\mathcal{L}^{2}_{0,\hat{X}^{\epsilon}(s)}\xi\left(s,\hat{X}^{\epsilon}(s),\hat{Y}^{\epsilon}(s)\right)
\right]ds\nonumber
\end{eqnarray}
Under our assumptions, the random variable $R_{2}(\epsilon,v)$ converges in $L^{2}$ to zero as $\epsilon,\delta\downarrow 0$ uniformly in $v\in\mathcal{A}$. Recalling the definitions of the controls $\bar{u}_{1},\bar{u}_{2}$ we get

\begin{eqnarray}
h(\hat{X}^{\epsilon}(T))-\bar{U}(t,x)&\geq&\left(\epsilon/\delta-\gamma\right)\int_{t}^{T}\left<\nabla_{x} \bar{U}\left(s,\hat{X}^{\epsilon}(s)\right),b\left(\hat{X}^{\epsilon}(s),\hat{Y}^{\epsilon}(s)\right)\right>ds\nonumber\\
& &
-\int_{t}^{T}\left<\bar{u}_{1}, v_{1}(s)-\bar{u}_{1}\right>\left(s,\hat{X}^{\epsilon}(s),\hat{Y}^{\epsilon}(s)\right)ds+\nonumber\\
& &-\int_{t}^{T}\left<\bar{u}_{2},v_{2}(s)-\bar{u}_{2}\right>\left(s,\hat{X}^{\epsilon}(s),\hat{Y}^{\epsilon}(s)\right)ds+\nonumber\\
& &+\frac{1}{2}\int_{t}^{T}\left\Vert \bar{u}_{1}\left(s,\hat{X}^{\epsilon}(s),\hat{Y}^{\epsilon}(s);2\right)\right\Vert^{2}ds+\frac{1}{2}\int_{t}^{T}\left\Vert \bar{u}_{2}\left(s,\hat{X}^{\epsilon}(s),\hat{Y}^{\epsilon}(s);2\right)\right\Vert^{2}ds\nonumber\\
& &+R_{1}(\epsilon,v)+R_{2}(\epsilon,v)\nonumber
\end{eqnarray}

Writing for notational convenience $\bar{u}_{i}(s)=\bar{u}_{i}\left(s,\hat{X}^{\epsilon}(s),\hat{Y}^{\epsilon}(s);2\right)$ for $i=1,2$, we get after some term rearrangement

\begin{eqnarray}
-\int_{t}^{T}\left[\left\Vert \bar{u}_{1}\left(s\right)\right\Vert^{2}+\left\Vert \bar{u}_{2}\left(s\right)\right\Vert^{2}\right]ds&\geq& \bar{U}(t,x)-h(\hat{X}^{\epsilon}(T))+\frac{1}{2}\int_{t}^{T}\left[\left\Vert \bar{u}_{1}\left(s\right)\right\Vert^{2}+\left\Vert \bar{u}_{2}\left(s\right)\right\Vert^{2}\right]ds\nonumber\\
& &-\int_{t}^{T}\left<\bar{u}_{1}\left(s\right), v_{1}(s)\right>ds-\int_{t}^{T}\left<\bar{u}_{2}\left(s\right),v_{2}(s)\right>ds+R(\epsilon,v)\label{Eq:UniformlyLogEfficient2Regime2_5}
\end{eqnarray}
where
$R(\epsilon,v)=R_{1}(\epsilon,v)+R_{2}(\epsilon,v)+\left(\epsilon/\delta-\gamma\right)\int_{t}^{T}\left<\nabla_{x}
\bar{U}\left(s,\hat{X}^{\epsilon}(s)\right),b\left(\hat{X}^{\epsilon}(s),\hat{Y}^{\epsilon}(s)\right)\right>ds$.
Since $\frac{\epsilon}{\delta}\rightarrow\gamma$ and
$R_{1}(\epsilon,v),R_{2}(\epsilon,v)$ converge in $L^{2}$ to zero as
$\epsilon,\delta\downarrow 0$, Condition \ref{Cond:ExtraRegReg2} implies that $R(\epsilon,v)$ converges
in $L^{2}$ to zero uniformly in $v\in\mathcal{A}$.

Inserting (\ref{Eq:UniformlyLogEfficient2Regime2_5}) into
(\ref{Eq:GoalRegime2_3}) gives us
\begin{eqnarray*}
-\epsilon\ln Q^{\epsilon}(t,x;\bar{u})&\geq&
\inf_{v\in\mathcal{A}}\mathrm{E}_{t,x,y}^{\epsilon}
\left[
\frac{1}{2}\int_{t}^{T}\left\Vert v(s)-\bar{u}\left(s,\hat{X}^{\epsilon}(s),\hat{Y}^{\epsilon}(s)\right)\right\Vert^{2}
ds+h(\bar{X}^{\epsilon, v-\bar{u}}(T)) \right.\nonumber\\
& &\hspace{1.5cm}
\left.+\bar{U}(t,x)+R(\epsilon,v)\right]\label{Eq:GoalRegime2_6}
\end{eqnarray*}

Set $\bar{v}(s)=v(s)-\bar{u}(s,\hat X(s),\hat Y(s))$. Since $\bar{v}\in\mathcal{A}$, the representation formula (\ref{Eq:VariationalRepresentation})
implies that
\begin{equation*}
\mathrm{E}\left[  \frac{1}{2}\int_{t}^{T}\left\Vert \bar{v}(s)\right\Vert
^{2}ds+h(\hat X(T))\right]  \geq-\epsilon\log\mathrm{E}\exp\left\{  -\frac
{1}{\epsilon}h(X^{\epsilon}(T))\right\}.
\end{equation*}
Recalling that $R(\epsilon, v)$ converges in $L^{2}$ to zero uniformly in $v\in\mathcal{A}$ as $\epsilon,\delta\downarrow 0$ and using statement (ii) of Theorem \ref{T:MainTheorem2} we get
\begin{align}
\liminf_{\epsilon\rightarrow0}-\epsilon\log
Q^{\epsilon}(t,x;\bar{u})  &
\geq\liminf_{\epsilon\rightarrow0}\inf_{\bar{v}\in\mathcal{A}}\mathrm{E}\left[
\frac{1}{2}\int_{t}^{T}\left\Vert \bar{v}(s)\right\Vert
^{2}ds+h(\hat
X(T))+R(\epsilon,v)\right]  +\bar{U}(t,x)\nonumber\\
&  \geq\liminf_{\epsilon\rightarrow
0}-\epsilon\log\mathrm{E}\exp\left\{
-\frac{1}{\epsilon}h(X^{\epsilon}(T))\right\}  +\bar{U}(t,x)\nonumber\\
&  = G(t,x)+\bar{U}(t,x).\label{Eq:LDPlowerBound}
\end{align}

This concludes the proof of the theorem.

\end{proof}
We conclude this subsection with the following remark. This remark
relaxes the requirement of a solution pair
$(\xi_{\gamma}(x,y,p),\bar{H}_{\gamma}(x,p))$ to the cell problem (\ref{Eq:CellProblemRegime2_1}) to
a subsolution pair. This can be useful in problems where solving the
cell problem is difficult even numerically.

\begin{remark}\label{R:SubsolutionReg2}
In the proof of the theorem, the definition of the cell problem (\ref{Eq:CellProblemRegime2_1}) was
only used in (\ref{Eq:UniformlyLogEfficient2Regime2_2_1}). However,
it is easy to see that the inequality in
(\ref{Eq:UniformlyLogEfficient2Regime2_2_1}) would be true if
instead of the solution pair to the cell problem, a subsolution pair
was used, i.e. a pair $(\xi_{\gamma}(x,y,p),\bar{H}_{\gamma}(x,p))$
such that
 $H_{\gamma}(x,y,p,\nabla_{y}\xi_{\gamma},0,\nabla^{2}_{y}\xi_{\gamma},0)\geq \bar{H}_{\gamma}(x,p)$ for all $y\in\mathcal{Y}$
 and $(x,p)\in\mathbb{R}^{m}\times\mathbb{R}^{m}$. So, one can seek for subsolution pairs $(\xi_{\gamma}(x,y,p),\bar{H}_{\gamma}(x,p))$ to (\ref{Eq:CellProblemRegime2_1})
such that $\xi_{\gamma}(x,y,p)$ is periodic in $y$ and
  $\bar{H}_{\gamma}(x,p)$ is concave in $p$.
\end{remark}

\subsection{Importance sampling for Regime 3.}\label{SS:IS_Regime3}
Finally, we study the construction of efficient importance samplings for Regime 3. The procedure here is similar to that of Regime 2. This is to be expected, since Regime $3$ is a limiting case of Regime $2$ obtained by setting $\gamma= 0$. Therefore, we shall only present the result omitting the proof, which follows as the proof of Theorem \ref{T:UniformlyLogEfficientReg2} for Regime 2. The statement for the existence and regularity of a pair $\left(\xi_{0}(x,y,p),\bar{H}_{0}(x,y,p)\right)$ satisfying (\ref{Eq:CellProblemRegime2_1}) with $\gamma=0$ is given in Section \ref{S:HJB}.

\begin{theorem}
\label{T:UniformlyLogEfficientReg3} Let $\{\left(X^{\epsilon}(s),
Y^{\epsilon}(s) \right),\epsilon>0\}$ be the solution to
(\ref{Eq:Main}) for $s\in[t_{0},T]$ with initial point $(x_{0},y_{0})$ at time $t_{0}$. Consider a bounded and continuous function
$h:\mathbb{R}^{m}\mapsto\mathbb{R}$ and assume Conditions
\ref{A:Assumption1}.
Let $\left(\xi_{0}(x,y,p),\bar{H}_{0}(x,p)\right)$ be a pair satisfying the cell problem (\ref{Eq:CellProblemRegime2_1}) with $\gamma=0$ and  $\bar{U}_{3}(s,x)$  be a classical subsolution according to
Definition \ref{Def:ClassicalSubsolution} with Hamiltonian
$\bar{H}_{0}(x,p)$ and assume Condition \ref{Cond:ExtraRegReg2} with
$\gamma=0$. Define the control $\bar{u}(s,x,y;3)=\left(\bar{u}_{1}(s,x,y;3),\bar{u}_{2}(s,x,y;3)\right)$ by
\begin{equation*}
\bar{u}(s,x,y;3)=\left(-\sigma^{T}(x,y)\nabla_{x}\bar{U}_{3}(s,x)-\tau_{1}^{T}(x,y)\nabla_{y}\xi_{0}\left(x,y,\nabla_{x}\bar{U}_{3}(s,x)\right),
-\tau_{2}^{T}(x,y)\nabla_{y}\xi_{0}\left(x,y,\nabla_{x}\bar{U}_{3}(s,x)\right)\right)\label{Eq:feedback_controlReg3}
\end{equation*}
Then the conclusion of Theorem \ref{T:UniformlyLogEfficient} holds, i.e.
\begin{equation*}
\liminf_{\epsilon\rightarrow0}-\epsilon\ln Q^{\epsilon}(t_{0},x_{0},y_{0};\bar{u}(\cdot;3))\geq
G_{3}(t_{0},x_{0})+\bar{U}_{3}(t_{0},x_{0}). \label{Eq:GoalRegime3Subsolution}%
\end{equation*}
\end{theorem}

Notice here that even though in the statement for the large deviations for Regime 3 (Theorem \ref{T:MainTheorem2}), we require that $g(x,y)=g(y)$ and $\tau_{i}(x,y)=\tau_{i}(y)$, in the statement of the related importance sampling lower bound we do not require that assumption. The reason is that in the proof of the importance sampling bound only the Laplace principle lower bound is used (compare with (\ref{Eq:LDPlowerBound})) and that holds with the $x-$dependence as well; see the second statement of Theorem \ref{T:MainTheorem2}.

\section{Connection with homogenization of Hamilton-Jacobi-Bellman equations.}\label{S:HJB}
It is evident from the calculations in Section \ref{S:IS} that there
is an implied relation of importance sampling for multiscale
problems and homogenization of a related class of HJB equations. In
this section we aim to make this connection clear. We only outline
the results that are relevant to the importance sampling results.
 We refer the interested reader to the literature of homogenization for Hamilton-Jacobi-Bellman equations for more detailed discussions,
e.g. \cite{AlvarezBardi2001,ArisawaLions,BuckdahnIchihara,Evans2,HorieIshii,LionsSouganidis}.

Let us define the function
\begin{equation*}
\theta^{\epsilon}(t,x,y)=\mathbb{E}_{t,x,y}\left[e^{-\frac{1}{\epsilon}h(X^{\epsilon}(T))}\right]
\end{equation*}
where $(X^{\epsilon},Y^{\epsilon})$ is the strong solution to the uncontrolled process (\ref{Eq:Main}) with initial point $(X^{\epsilon}(t),Y^{\epsilon}(t))=(x,y)$. A straightforward computation shows that the function
\begin{equation*}
G^{\epsilon}(t,x,y)=-\epsilon \ln \theta^{\epsilon}(t,x,y)
\end{equation*}
solves the Hamilton-Jacobi-Bellman equation
\begin{eqnarray}
\partial_{t}G^{\epsilon}+H_{\epsilon/\delta}\left(x,y,\nabla_{x}G^{\epsilon},\frac{\nabla_{y}G^{\epsilon}}{\delta},\epsilon\nabla^{2}_{x}G^{\epsilon},
 \frac{\nabla^{2}_{y}G^{\epsilon}}{\delta},\nabla_{x}\nabla_{y}G^{\epsilon} \right)&=&0\nonumber\\
G^{\epsilon}(T,x,y)&=&h(x)\nonumber
\end{eqnarray}
where the Hamiltonian $H_{\epsilon/\delta}$ is defined as in (\ref{Eq:CellProblemRegime2_0}) with $\epsilon/\delta$ in place of $\gamma$. Under Conditions $\ref{A:Assumption1}$ and $\ref{A:Assumption2}$ we have the following.

\begin{itemize}
\item{In the case of \textit{Regime $1$}, we have that
$G^{\epsilon}(t,x,y)$ converges uniformly in compact subsets of
$[0,T]\times\mathbb{R}^{m}\times\mathbb{R}^{d-m}$ to the unique
bounded and continuous viscosity solution of (\ref{Eq:HJBequation2})
with effective Hamiltonian given by (\ref{Eq:ControlFormHJB}). We
refer the reader to \cite{BuckdahnIchihara,AlvarezBardi2001} for
details.}

\item{In the case of \textit{Regime $2$} the effective  equation has again the form (\ref{Eq:HJBequation2}) but the effective Hamiltonian is given by the unique constant $\bar{H}_{\gamma}$ such that the periodic cell problem (\ref{Eq:CellProblemRegime2_1}) has a unique (up to an additive constant) periodic solution $\xi_{\gamma}\in \mathcal{C}^{2}(\mathbb{R}^{d-m})$ (see Theorem II.2 in \cite{ArisawaLions}). Under our assumptions, the effective Hamiltonian $\bar{H}_{\gamma}(x,p)$ is continuous in $x$ and concave in $p$ (see Propositions $11$ and $12$ in \cite{AlvarezBardi2001}).}

\item{In the case of \textit{Regime $3$} the effective  equation has again the form (\ref{Eq:HJBequation2}) but the effective Hamiltonian is given by the unique constant $\bar{H}_{0}$ such that the periodic cell problem (\ref{Eq:CellProblemRegime2_1}) with
 $\gamma=0$ has a Lipschitz continuous periodic solution $\xi_{0}$ (see \cite{AlvarezBardi2001, BardiDolcetta}). Again, under our assumptions, the effective Hamiltonian $\bar{H}_{0}(x,p)$ is continuous in $x$ and concave in $p$ (see Propositions $3$ in \cite{AlvarezBardi2001}).}
\end{itemize}
In regards to how these general results apply to importance sampling for multiple scale problems, we have the following remark.

\begin{remark}
In the context of importance sampling, we observe two things:
\begin{enumerate}
\item{the subsolutions that we are considering are subsolutions to
the corresponding limiting HJB equations, and}
\item{the cell problem arising in homogenization of HJB equations enters in the formulation of the importance sampling scheme in each regime. }
\end{enumerate}
These
imply that in Monte Carlo simulation for multiscale problems both
the local information described by the corresponding cell problem
and the homogenized information that is described by the solution to
the HJB equation, enter the asymptotically optimal change of
measure. As it is demonstrated in the numerical simulations
presented in \cite{DupuisSpiliopoulosWang}, neglecting the local
information and basing the simulation only on the homogenized
information can lead to estimators that perform poorly in the small
noise regime.
\end{remark}

\section{Examples}\label{S:Examples}
In this section we present some simple examples from the existing literature to illustrate how our calculations look like. We consider two examples. The first one is the  first order Langevin equation.
As we said in the introduction this model can be used to model rough energy landscapes motivated by applications in chemistry; see also
 \cite{LifsonJackson,MondalGhosh,SavenWangWolynes, DupuisSpiliopoulosWang2, Zwanzig}. This model was extensively discussed in  \cite{DupuisSpiliopoulosWang,DupuisSpiliopoulosWang2} and the theory
 was also demonstrated by simulation results. We recall the formulas here for completeness for this particularly important example. The second example is related to short time asymptotics for processes that depend on another fast mean reverting process. Models of this nature appear in mathematical finance in the context of fast mean reverting stochastic volatility models, e.g., \cite{FengFouqueKumar}. Assuming that we want to estimate
\begin{equation*}
\theta(\epsilon)=\mathrm{E}[e^{-\frac{1}{\epsilon}h(X^{\epsilon}(T))}|X^{\epsilon}(0)=x_{0}, Y^{\epsilon}(0)=y_{0}]
\end{equation*}
for a given function $h(x)$ and a given corresponding subsolution $\bar{U}$, we also provide the control that attains the desired bounds in Theorems \ref{T:UniformlyLogEfficientReg1}, \ref{T:UniformlyLogEfficientReg2} and \ref{T:UniformlyLogEfficientReg3}.

\subsection{The first order Langevin equation}\label{SS:Example1}
We consider
the first order Langevin  equation
\begin{equation}
dX^{\epsilon}(s)=\left[  -\frac{\epsilon}{\delta}\nabla Q\left(
\frac{X^{\epsilon}(s)}{\delta}\right)  -\nabla V\left(  X^{\epsilon
}(s)\right)  \right]  dt+\sqrt{\epsilon}\sqrt{2D}dW(s),\hspace{0.2cm}%
X^{\epsilon}(0)=x_{0}. \label{Eq:LangevinEquation2}%
\end{equation}
To connect to the notation of the general model (\ref{Eq:Main}), this corresponds to
\[
f(x,y)=b(x,y)=-\nabla Q(y),~~~g(x,y)=c(x,y)=-\nabla V(x),~~~\tau_{1}(x,y)=\sigma(x,y)=\sqrt{2D}, ~~~\tau_{2}(x,y)=0.
\]

Let us consider the case of Regime $1$. The invariant distribution associated to the operator $\mathcal{L}^{1}$ is the Gibbs distribution (independent of $x$)
\[
\mu(dy) = \frac{1}{L}e^{-\frac{Q(y)}{D}}dy,\hspace{0.2cm}%
~~L=\int_{\mathcal{Y}}e^{-\frac{Q(y)}{D}}dy.
\]
Moreover, Condition \ref{A:Assumption2} is trivially satisfied. In
dimension $1$, an easy computation shows that the action functional
takes the following explicit form
\begin{equation*}
S_{0T}(\phi)=%
\begin{cases}
\displaystyle{\frac{1}{2}\int_{0}^{T}\frac{1}{q}[\dot{\phi}(s)-r(\phi
(s))]^{2}ds} & \text{if }\phi\in\mathcal{AC}([0,T];\mathbb{R})\text{ and }%
\phi(0)=x_{0}\\
+\infty & \text{otherwise},
\end{cases}
\label{ActionFunctional1_1}%
\end{equation*}
where
\[
r(x)=-\frac{\lambda^{2}V^{\prime}(x)}{L\hat{L}},\hspace{0.5cm}q=\frac
{2D\lambda^{2}}{L\hat{L}}%
\]
and
\[
L=\int_{\mathcal{Y}}e^{-\frac{Q(y)}{D}}dy,\hspace{0.5cm}\hat{L}=\int
_{\mathcal{Y}}e^{\frac{Q(y)}{D}}dy.
\]

In addition, we can also compute the optimal change of measure in regards to the importance sampling problem.
Given a classical subsolution $\bar{U}$, the importance sampling control that
appears in Theorem \ref{T:UniformlyLogEfficientReg1} takes the form
\begin{equation*}
\bar{u}(s,x,y;1)=\left(-\frac{\sqrt{2D}\lambda}{\hat{L}}e^{\frac{Q(y)}{D}}\partial_{x}\bar{U}%
(s,x),0\right). \label{Eq:OptimalControlExample1}%
\end{equation*}
The choice of the subsolution $\bar{U}$ according to Definition \ref{Def:ClassicalSubsolution} depends on the terminal cost of interest $h(x)$. See also \cite{DupuisSpiliopoulosWang,DupuisSpiliopoulosWang2} for some particular examples with specific choices of subsolutions $\bar{U}(s,x)$.

\subsection{Short time asymptotics and fast mean reversion.}\label{SS:Example2}

Next we consider a particular system of slow-fast motion, where the fast motion is a fast mean reverting process. The slow motion appears due to the interest in short time asymptotics.
In particular, let us consider the system in $1+1$ dimension

\begin{eqnarray}
dX(s)&=&h\left(Y(s)\right)ds+\sigma\left(Y(s)\right)
dW(s), \label{Eq:SVModela}\\
dY(s)&=& \frac{1}{\delta^{2}}\left(m-Y(s)\right)   ds+\frac{1}{\delta}\left[
\rho dW(s)+\sqrt{1-\rho^{2}}dB(s)\right]\nonumber
\end{eqnarray}
where $0<\delta\ll 1$ is the fast mean reversion parameter, $m\in\mathbb{R}$ and $\rho\in[-1,1]$ is the correlation between the noise of the $X$ and $Y$ process. Assume that we are interested in short time asymptoptics. Then it is convenient to change time $s\mapsto \epsilon s$ with $0<\epsilon\ll 1$. Writing the system under the new timescale, we obtain $\left\{\left(X^{\epsilon}(s), Y^{\epsilon}(s)\right), s\in[0,T]\right\}$ as the unique strong solution to:

\begin{eqnarray}
dX^{\epsilon}(s)&=&\epsilon h\left(Y^{\epsilon}(s)\right)ds+\sqrt{\epsilon}%
\sigma\left(Y^{\epsilon}(s)\right)
dW(s), \label{Eq:SVModel}\\
dY^{\epsilon}(s)&=& \frac{\epsilon}{\delta^{2}}\left(m-Y^{\epsilon}(s)\right)   ds+\frac{\sqrt{\epsilon}}{\delta}\left[
\rho dW(s)+\sqrt{1-\rho^{2}}dB(s)\right]\nonumber
\end{eqnarray}

Both components $(X,Y)$ take values in $\mathbb{R}$. We supplement the system with initial condition $(X^{\epsilon}(0),Y^{\epsilon}(0))=(x_{0},y_{0})$. To connect to the notation of the general model (\ref{Eq:Main}), this corresponds to
\[
b(x,y)=0,~~~c^{\epsilon}(x,y)=\epsilon h(y), ~~~\sigma(x,y)=\sigma(y),
 \]
 \[f(x,y)=m-y,~~~ g(x,y)=0, ~~~\tau_{1}(x,y)=\rho, ~~~\tau_{2}(x,y)=\sqrt{1-\rho^{2}}.
\]

Of course, this system violates the periodicity assumption. However
due to the mean reverting feature of the fast motion, the
conclusions hold in this case as well.

 In the next subsections we see the form of the large deviations action functional and of the control that defines the asymptotically optimal change of measure for all three regimes.
\subsubsection{The case of Regime 1.}
A simple computation shows that the only possible solution to cell
problem (\ref{Eq:CellProblem}) is the zero solution (this is because
$b=0$). Also, it is easy to see that the invariant measure
corresponding to the operator $\mathcal{L}^{1}$ is independent of
$x$ and can be explicitly computed, taking the form
\begin{equation*}
\mu(dy)=\frac{1}{\sqrt{\pi}} e^{-(y-m)^{2}}dy
\end{equation*}
Then this implies that the formula for the action functional (Theorem \ref{T:MainTheorem3}) becomes
\begin{equation*}
S_{0T}(\phi)=%
\begin{cases}
\displaystyle{\frac{1}{2}\int_{0}^{T}\frac{1}{q}|\dot{\phi}(s)|^{2}ds} & \text{if }\phi\in\mathcal{AC}([0,T];\mathbb{R})\text{ and }%
\phi(0)=x_{0}\\
+\infty & \text{otherwise},
\end{cases}
\end{equation*}
where $q=\int_{\mathcal{Y}}\sigma^{2}(y)\mu(dy)$. Given a classical subsolution $\bar{U}$, the importance sampling control that
appears in Theorem \ref{T:UniformlyLogEfficientReg1} takes the form
\begin{equation*}
\bar{u}(s,x,y;1)=\left(-\sigma(y)\partial_{x}\bar{U}(s,x),0\right).
\end{equation*}
As in the previous example, the choice of the subsolution $\bar{U}$ according to Definition \ref{Def:ClassicalSubsolution} depends on the terminal cost of interest $h(x)$.
\subsubsection{The case of Regime 2.}
 The situation here is more complicated because the infimization problem that appears in the definition
of the local rate function, Theorem \ref{T:MainTheorem2}, does not
necessarily have a closed form solution as it had for Regime $1$.
However, due to the one-dimensionality aspect of the problem we can
still do some algebraic computations. A simple algebra shows that the formula
for the action functional (Theorem \ref{T:MainTheorem2}) becomes
\begin{equation*}
S_{0T}(\phi)=%
\begin{cases}
\displaystyle{\frac{1}{2}\int_{0}^{T}L_{2}(\phi_{s},\dot{\phi}_{s})ds} & \text{if }\phi\in\mathcal{AC}([0,T];\mathbb{R})\text{ and }%
\phi(0)=x_{0}\\
+\infty & \text{otherwise},
\end{cases}
\end{equation*}
where
\begin{equation*}
L_{2}(x,\beta)=\inf_{v\in\mathcal{A}_{x,\beta}^{2}}\left\{
\frac{1}{2}\int_{\mathcal{Y}}| v(y)|^{2}\mu_{v}(dy)\right\}.
\end{equation*}
with
\begin{equation*}
\mu_{v}(dy)=\frac{1}{L}e^{\int_{1}^{y}\left[2\gamma(m-z)+2 \rho v(z)\right]dz} dy,\qquad L=\int_{\mathcal{Y}}e^{\int_{1}^{y}\left[2\gamma(m-z)+2 \rho v(z)\right]dz}dy
\end{equation*}
\begin{equation*}
\mathcal{A}_{x,\beta}^{2}  =\left\{  v(\cdot):\mathcal{Y}\mapsto
\mathbb{R},\beta=\int_{\mathcal{Y}}\sigma(y)v(y)\mu_{v}(dy)\right\}.
\end{equation*}

Notice that the invariant measure $\mu(dy)$ and the control $v$ decouple when $\rho=0$. In the general case $\rho\in[-1,1]$,  the equation for the related cell problem (\ref{Eq:CellProblemRegime2_1}) takes the form
\begin{equation*}
\frac{\gamma}{2}\xi_{\gamma}''(y)+\left[\gamma(m-y)-\sigma(y)\rho p\right]\xi_{\gamma}'(y)-\frac{1}{2}\left(\xi_{\gamma}'(y)\right)^{2}-\frac{1}{2}\sigma^{2}(y)p^{2}=\bar{H}_{\gamma}(p)
\end{equation*}

There is a unique pair
$(\xi_{\gamma}(y),\bar{H}_{\gamma}(p))$ satisfying this equation
such that $\xi_{\gamma}(y)\in W^{1}_{loc}$, see
\cite{ArisawaLions,KaiseSheu}. Notice that for this model, the
solution $(\xi_{\gamma}(y),\bar{H}_{\gamma}(p))$ to the cell problem
is independent of the slow motion $x$. Obtaining closed form
solutions to such equations is difficult in principle, especially because we are interested in pairs $(\xi_{\gamma}(y),\bar{H}_{\gamma}(p))$.
Numerical methods such as the ones developed in \cite{CamilliMarchi,GomesOberman} will be useful here. Notice also that by Remark \ref{R:SubsolutionReg2} appropriate
subsolution pairs suffice.  We plan to return to these issues in detail in a future work.

Given sufficient smoothness such that Theorem \ref{T:UniformlyLogEfficientReg2} is applicable and a classical subsolution $\bar{U}$ (depending on the choice of the terminal cost $h(x)$), the importance sampling control that
appears in Theorem \ref{T:UniformlyLogEfficientReg2} takes the form
\begin{equation*}
\bar{u}(s,x,y;2)=\left(-\sigma(y)\partial_{x}\bar{U}(s,x)-\rho \partial_{y}\xi_{\gamma}\left(y,\partial_{x}\bar{U}(s,x)\right), -\sqrt{1-\rho^{2}}\partial_{y}\xi_{\gamma}\left(y,\partial_{x}\bar{U}(s,x)\right)\right).
\end{equation*}

\subsubsection{The case of Regime 3.}
It turns out that we can make some explicit computations here. To
simplify things we will assume for brevity that $\rho=1$. With these assumptions we get that $\tau_{1}(x,y)=1$
and $\tau_{2}(x,y)=0$. Assume that $\sigma\in L^{1}(\mathcal{Y})$ and that $\int_{\mathcal{Y}}\sigma(y)dy\neq 0$. A
straightforward computation shows that the local rate function takes
the form
\begin{equation*}
L_{3}(x,\beta)=\inf_{v}\left\{
\frac{1}{2}\int_{\mathcal{Y}}|v(y)|^{2}\frac{\beta}{v(y)\int_{\mathcal{Y}}\sigma(y)dy}dy: \int_{\mathcal{Y}}\frac{\beta}{v(y)\int_{\mathcal{Y}}\sigma(y)dy}dy=1\right\}.
\end{equation*}
This problem can be solved explicitly yielding
\begin{equation*}
S_{0T}(\phi)=%
\begin{cases}
\displaystyle{\frac{1}{2}\int_{0}^{T}|\dot{\phi}_{s}|^{2}\left(\int_{\mathcal{Y}}\sigma(y)dy\right)^{-2}ds} & \text{if }\phi\in\mathcal{AC}([0,T];\mathbb{R})\text{ and }%
\phi(0)=x_{0}\\
+\infty & \text{otherwise},
\end{cases}
\end{equation*}

The equation for the related cell problem (\ref{Eq:CellProblemRegime2_1}) with $\gamma=0$ takes the particular simple form
\begin{equation*}
-\sigma(y)p\xi_{0}'(y)-\frac{1}{2}\sigma^{2}(y)p^{2}-\frac{1}{2}\left(\xi_{0}'(y)\right)^{2}=\bar{H}_{0}(p)
\end{equation*}
This has the form of first order Bellman equation with quadratic Hamiltonian. Such equations have been studied in the literature and our assumptions guarantee that there are pairs
$(\xi_{0},\bar{H}_{0})$ such that $\xi_{0}$ is a continuous viscosity solution  when $\bar{H}_{0}\geq \bar{H}_{0}^{*}$ where $\bar{H}_{0}^{*}$ is a critical value. We refer the
interested reader to \cite{KaiseSheu2} for an extensive discussion on this.

If $\sigma(y)$ is periodic in $y$, say with period $\lambda=1$, then $\mathcal{Y}=\mathbb{T}=[0,1]$ and we look for a periodic solution $\xi_{0}(y)$. It turns out that the Bellman equation
can then be solved explicitly yielding

\begin{equation*}
\xi_{0}(y,p)= p\left(y\int_{0}^{1}\sigma(w)dw-\int_{0}^{y}\sigma(w)dw\right), \qquad \bar{H}_{0}(p)=-\frac{1}{2}p^{2}\left(\int_{0}^{1}\sigma(y)dy\right)^{2}
\end{equation*}

Thus, indeed $(\xi_{0}(y,p),\bar{H}_{0}(p))$ satisfy the assumptions of Theorem \ref{T:UniformlyLogEfficientReg3}.
Given a classical subsolution $\bar{U}$ (depending on the choice of the terminal cost $h(x)$), the importance sampling control that
appears in Theorem \ref{T:UniformlyLogEfficientReg3} takes the particularly simple form
\begin{equation*}
\bar{u}(s,x,y;3)=\left(-\sigma(y)\partial_{x}\bar{U}(s,x)-\partial_{y}\xi_{0}(y,\partial_{x}\bar{U}(s,x)), 0\right)=\left(-\left[\int_{0}^{1}\sigma(w)dw\right] \partial_{x}\bar{U}(s,x),0\right).
\end{equation*}

\section{Conclusions}\label{S:Conclusions}
In this paper we have developed the large deviations theory and a rigorous mathematical framework for the importance sampling theory for systems of slow-fast motion
 like (\ref{Eq:Main}).  All the possible cases of interaction of fast motion and intensity of the noise are considered. The asymptotic performance of the proposed schemes are
 in terms of appropriate subsolutions to related HJB equations and in terms of appropriate "cell problems". Straightforward adaptation of importance sampling schemes from
 standard diffusions without multiscale features lead to poor results in the multiscale setting. We have shown how the problem can be dealt with in the general multidimensional
 setting for fully dependent systems of slow-fast motion, when the fast motion is periodic.


\begin{thebibliography}{99}


\bibitem {AlvarezBardi2001}O. Alvarez and M. Bardi, Viscosity solutions
methods for singular perturbations in deterministic and stochastic control,
\textit{SIAM Journal on Control and Optimization,}  40(4), (2001), pp. 1159-1188.


\bibitem{ArisawaLions}
M. Arisawa, P.-L Lions, On ergodic stochastic control, Comm. Partial Differential Equations, 23, (1998), pp. 2187-2217.


\bibitem {Baldi}P. Baldi, Large deviations for diffusions processes with
homogenization and applications, \emph{Annals of Probability}, 19(2),
(1991), pp. 509--524.

\bibitem {BardiDolcetta}M. Bardi and I. Capuzzo Dolcetta, \emph{Optimal
Control and Viscosity Solutions of Hamilton Jacobi Bellman Equations},
Birk\"{a}user, Boston, 1997.

\bibitem {BensoussanFrehse}A. Bensoussan, J. Frehhse, On Bellman equations
of ergodic control in $\mathbb{R}^{n}$, \textit{Journal f\"{u}r die Reine und Angewandte Mathematik}, Vol. 429, (1992),
pp. 125-160.

\bibitem {BLP}A. Bensoussan, J.L. Lions and G. Papanicolaou, \emph{Asymptotic
Analysis for Periodic Structures}, Vol 5, Studies in Mathematics and its
Applications, North-Holland Publishing Co., Amsterdam, 1978.



\bibitem {BorkarGaitsgory}V. Borkar, V. Gaitsgory, Averaging of singularly
perturbed controlled stochastic differential equations, \textit{Applied
Mathematics and Optimization}, Vol 56, No. 2, (2007), pp. 169-209.


\bibitem {BoueDupuis}M. Bou\'{e} and P. Dupuis, A variational representation
for certain functionals of Brownian motion, \emph{Annals of
Probability}, 26(4), (1998), pp. 1641-1659.

\bibitem {BuckdahnIchihara}R. Buckdahn and N. Ichihara, Limit theorem for
controlled backward SDEs and homogenization of Hamilton-Jacobi-Bellman
equations, \emph{Applied Mathematics and Optimization}, 51 (2005), pp. 1-33.


\bibitem{CamilliMarchi}
F. Camilli and C. Marchi, Rates of convergence in periodic homogenization of fully nonlinear uniformly elliptic PDEs, \emph{Nonlinearity}, 22 (2009) pp. 1481-1498.

\bibitem {CrandallIshiiLions1992}M. G. Crandall, H. Ishii and P.-L. Lions,
User's guide to viscosity solutions of second order partial differential
equations, \emph{Bull. Amer. Math. Soc.,} (N.S.), 27(1), (1992), pp. 1-67.

\bibitem {DeanDupuis}T. Dean and P. Dupuis, Splitting for rare event
simulation: a large deviation approach to design and analysis,
\emph{Stochastic Processes and their Applications}, 119 (2009), pp. 562-587.

\bibitem {DupuisEllis}P. Dupuis and R.S. Ellis, \emph{A Weak Convergence
Approach to the Theory of Large Deviations}, John Wiley \& Sons, New York, 1997.

\bibitem {DupuisSpiliopoulos}P. Dupuis and K. Spiliopoulos, Large deviations
for multiscale problems via weak convergence methods, \emph{Stochastic
Processes and their Applications}, (2012), Vol. 122, pp. 1947-1987.

\bibitem{DupuisSpiliopoulosWang}
P. Dupuis, K. Spiliopoulos, H. Wang.
\newblock Importance sampling for multiscale diffusions,
\newblock \emph{SIAM Journal on Multiscale Modeling and Simulation}, Vol. 12, No. 1, (2012), pp. 1-27.

\bibitem{DupuisSpiliopoulosWang2}
P. Dupuis, K. Spiliopoulos, H. Wang.
\newblock Rare Event Simulation in Rough Energy Landscapes.
\newblock {\em 2011 Winter Simulation Conference}.
\newblock appeared.

\bibitem {DupuisWang}P. Dupuis and H. Wang, Importance sampling, large
deviations and differential games, \emph{Stochastics and Stochastics Reports,}
76, (2004), pp. 481-508.

\bibitem {DupuisWang2}P. Dupuis and H. Wang, Subsolutions of an Isaacs
equation and efficient schemes for importance sampling, \emph{Mathematics of
Operations Research}, 32(3), (2007), pp. 723-757.

\bibitem {EithierKurtz}S.N. Eithier and T.G. Kurtz, \emph{Markov Processes:
Characterization and Convergence}, John Wiley \& Sons, New York, 1986.

\bibitem {Evans2}L. Evans, Periodic homogenization of certain fully nonlinear
partial differential equations, \emph{Proc. Roy. Soc. Edinburgh
Section A}, 120 (1992), pp. 245-265.

\bibitem {FlemingSoner}W.H. Fleming and H.M. Soner, \emph{Controlled Markov
Processes and Viscosity Solutions}, Springer, 2nd Ed., 2006.





\bibitem{FengFouqueKumar}
J. Feng, J.-P. Fouque, and R. Kumar, Small-time asymptotics for fast mean-reverting stochastic volatility models, \emph{Annals of Applied Probability}, Vol. 22(4), (2012), pp. 1541-1575.

\bibitem {FS}M. Freidlin and R. Sowers, A comparison of homogenization and
large deviations, with applications to wavefront propagation, \emph{Stochastic
Process and their Applications}, 82(1), (1999), pp. 23--52.


\bibitem{GomesOberman}
D. Gomes D and A. Oberman,  Computing the effective Hamiltonian using a variational approach \emph{SIAM J.
Control Optimization}, 43, 2004, pp. 792-812.


\bibitem {HorieIshii}K. Horie and H. Ishii, Simultaneous effects of
homogenization and vanishing viscosity in fully nonlinear elliptic equations,
\emph{Funkcialaj Ekvaciaj}, 46(1), (2003), pp. 63-88.



\bibitem {KaiseSheu}H. Kaise, S.J. Sheu, On the structure of solutions of
ergodic type Bellman equation related to risk-sensitive control, \textit{Annals of Probability}, Vol. 34, No. 1, (2006), pp. 284-320.

\bibitem {KaiseSheu2}H. Kaise, S.J. Sheu, Ergodic Type Bellman Equations of First Order
with Quadratic Hamiltonian, \textit{Applied Mathematics Optimization}, Vol. 59, No. 1, (2009), pp. 37-73.


\bibitem {Kushner}H. J. Kushner, \textit{Weak Convergence Methods and
Singularly Perturbed Stochastic Control and Filtering Problems},
Birkh\"{a}user, Boston-Basel-Berlin, (1990).

\bibitem {Kushner1}H. J. Kushner, Large deviations for two-time-scale
diffusions with delays, \emph{Applied Mathematics and Optimization}, Vol. 62, Issue 3, (2009), pp. 295-322.


\bibitem {LifsonJackson}S. Lifson and J.L. Jackson, On the self-diffusion of
ions in a polyelectrolyte solution, \emph{Journal of Chemical Physics},
36, (1962), pp. 2410-2414.

\bibitem {LionsSouganidis}P.-L. Lions and P.E. Souganidis, Homogenization of
degenerate second-order PDE in periodic and almost periodic environments and
applications, \emph{Ann. Inst. H. Poincar\'{e} Anal. Non Lin\'{e}aire}, 22(5), (2005), pp. 667-677.

\bibitem {Lipster}R. Lipster, Large deviations for two scaled diffusions,
\textit{Probability Theory and Related Fields}, Vol. 106, No. 1, (1996), pp. 71-104.

\bibitem {MondalGhosh}D. Mondal, P.K. Ghosh and D.S. Ray, Noise-induced
transport in a rough racket potential, \emph{Journal of Chemical Physics}, 130, (2009), pp. 074703.1-074703.7.

\bibitem {PardouxVeretennikov1}E. Pardoux, A.Yu. Veretennikov, On Poisson
equation and diffusion approximation 2, \textit{Annals of Probability}, Vol.
31, No. 3, (2003), pp. 1166-1192.


\bibitem {PS}G.A. Pavliotis and A.M. Stuart, \emph{Multiscale Methods:
Averaging and Homogenization}, Springer, 2007.

\bibitem {SavenWangWolynes}J.G. Saven, J. Wang and P.G.Wolynes, Kinetics of
protein folding: The dynamics of globally connected rough energy landscapes
with biases, \emph{Journal of Chemical Physics}, 101(12), (1994), pp. 11037-11043.

\bibitem {Veretennikov}A. Yu. Veretennikov, On large deviations in the
averaging principle for {SDEs} with a ``full dependence'', correction,
{\tt arXiv:math/0502098v1 [math.PR]} (2005). Initial
article in \textit{Annals of Probability}, Vol. 27, No. 1, (1999), pp. 284-296.

\bibitem {VeretennikovSPA2000}A. Yu. Veretennikov, On large deviations for
SDEs with small diffusion and averaging, \textit{Stochastic Processes and
their Applications}, Vol. 89, Issue 1, (2000), pp. 69-79.


\bibitem {Zwanzig}R. Zwanzig, Diffusion in a rough potential, \emph{Proc.
Natl. Acad. Sci. USA}, 85, (1988), pp. 2029-2030.

\end{thebibliography}
\end{document}